\documentclass{amsart}

\usepackage{amsmath,amssymb,mathrsfs}
\usepackage{graphicx,cite}
\usepackage{subfigure}

\usepackage{amssymb}
\usepackage{amsthm}
\usepackage{color}
\usepackage{algorithm}
\usepackage{algorithmic}
\usepackage{amsmath}
\usepackage{multirow}
\usepackage{epstopdf}
\usepackage{tikz}
\input{undertilde}
\usetikzlibrary{snakes}

\usepackage{verbatim}
\usepackage{ulem}

\newtheorem{theorem}{Theorem}[section]

\newtheorem{lemma}[theorem]{Lemma}

\newtheorem{rema}[theorem]{Remark}

\numberwithin{equation}{section}

\begin{document}

\title[Computation of Elasticity Transmission Eigenvalues]{Computation of Transmission Eigenvalues for Elastic Waves}

\author{Xia Ji}
\address{LSEC, Academy of Mathematics and System Sciences,
Chinese Academy of Sciences, Beijing, 100190, China.}
\email{jixia@lsec.cc.ac.cn}
\author{Peijun Li}
\address{Department of Mathematics, Purdue University, West Lafayette, IN
47907, USA.}
\email{lipeijun@math.purdue.edu}
\author{Jiguang Sun}
\address{Department of Mathematical Sciences, Michigan Technological University, Houghton, MI 49931, USA.}
\email{jiguangs@mtu.edu}

\thanks{The research of X. Ji is partially supported by the National Natural
Science Foundation of China with Grant Nos. 11271018 and 91630313, and National
Center for Mathematics and Interdisciplinary Sciences, Chinese Academy of
Sciences. The research of P. Li was supported in part by the NSF grant
DMS-1151308. The research of J. Sun was supported in part by the NSF grant DMS-1521555.}



\begin{abstract}
The goal of this paper is to develop numerical methods computing a few smallest elasticity transmission
eigenvalues, which are of practical importance in inverse scattering theory. The
problem is challenging since it is nonlinear, non-self-adjoint, and of fourth
order. We construct a nonlinear function whose values are generalized
eigenvalues of a series of self-adjoint fourth order problems. The roots of the
function are the transmission eigenvalues. Using an $H^2$-conforming finite
element for the self-adjoint fourth order eigenvalue problems,
we employ a secant method to compute the roots of the nonlinear function. The
convergence of the proposed method is proved. In addition, a mixed finite
element method is developed for the purpose of verification. Numerical examples
are presented to verify the theory and demonstrate the effectiveness of the two methods.
\end{abstract}

\maketitle

\section{Introduction}

Transmission eigenvalues have important applications in inverse scattering theory. For example,
they can be used to obtain useful information on the physical properties of the
scattering targets \cite{CakoniEtal2010IP,Sun2011IP}. In this paper, we consider the interior transmission
eigenvalue problem for elastic waves. Similar to the cases of acoustic
and electromagnetic waves, the elasticity transmission eigenvalue (ETE) problem
plays a critical role in the qualitative reconstruction methods for
inhomogeneous media. There are only a few theoretical studies on the ETE
problem \cite{Charalambopoulos2002JE, CharalambopoulosAnagnostopoulos2002JE,
BellisGuzina2010JE, BellisCakoniGuzina2013IMAAM}. It is shown in
\cite{BellisCakoniGuzina2013IMAAM} that there exists a countable set of
elasticity transmission eigenvalues under suitable conditions on elastic tensors
and mass densities. 

Numerical methods for the acoustic transmission eigenvalues have been developed
by many researchers recently \cite{ColtonMonkSun2010, Sun2011SIAMNA,
JiSunTurner2012ACMTOM, AnShen2013JSC, Kleefeld2013IP, JiSunXie2014JSC,
CakoniMonkSun2014CMAM, LiEtal2014JSC, YangBiLiHan2016}.
However, there exist much fewer papers \cite{MonkSun2012SIAMSC, SunXu2013IP,
HuangHuangLin2015SIAMSC} for the
electromagnetic transmission eigenvalue problems. It is highly non-trivial to
develop finite element methods for the transmission eigenvalue problems in
general since the problem is nonlinear and nonself-adjoint \cite{SunZhou2016}.
Although out of the scope of the current paper, it
is useful to point out that the finite element discretization usually leads to
non-Hermitian matrix eigenvalue problems. It is challenging to compute
(interior) generalized eigenvalues for non-Hermitian matrices. In particular,
when the size of matrices is large and there is no spectrum information,
classical methods in numerical linear algebra would fail. New methods have
emerged to treat such difficult problems \cite{HuangEtal2016JCP,
HuangEtal2017arXiv}.

The goal of this paper is to develop effective numerical methods to compute
a few smallest real transmission eigenvalues, which can be used to estimate
material property of the elastic body (see,
e.g., \cite{Sun2011IP}). Unlike the classical Laplacian eigenvalue problem or
the biharmonic eigenvalue problem, the transmission eigenvalue problem is
nonlinear and nonself-adjoint. To overcome this issue, we reformulate the
problem as a combination of a nonlinear function and a series of fourth order
self-adjoint eigenvalue problems. Specifically, the ETE is first written
as a nonlinear fourth order problem, which turns out to be a quadratic
eigenvalue problem. To avoid dealing with the nonself-adjointness directly, we construct a nonlinear function whose roots are the
elasticity transmission eigenvalues. The values of the nonlinear function are
generalized eigenvalues of self-adjoint coercive fourth order
problems, which can be treated using classical $H^2$-conforming finite
elements. A secant based iterative method is adopted to compute the roots of the
nonlinear function. In addition, we give a mixed method using the
Lagrange elements for the purpose of verification. 

The current paper, to the
authors' knowledge, is the first numerical study on the ETE.
We hope that it can attract more numerical analysts to this interesting and
challenging topic.
The rest of the paper is organized as follows. In Section 2, we introduce
the elasticity transmission eigenvalue problem and derive a quadratic eigenvalue
problem based on a fourth order partial differential equation. 
To avoid direct treatment of the nonlinearity and nonself-adjointness, the
problem is decomposed into a nonlinear function and a series of related linear
self-adjoint fourth order eigenvalue problems. The values of the nonlinear
function are generalized eigenvalues of the fourth order problems.
The roots of the nonlinear function are transmission eigenvalues.
$H^2$-conforming Argyris element for the fourth order problems is presented in Section 3. A secant based iterative method is used in
Section 4 to compute roots of the nonlinear function. Section 5 introduces a mixed finite element method for verification. 
Numerical experiments are presented in Section 6. The paper is concluded with some discussion and future works in Section 7.

\section{The elasticity transmission eigenvalue problem}

Let $\boldsymbol x=(x, y)^\top\in\mathbb R^2$ and $D\subset\mathbb{R}^2$ be a bounded Lipschitz domain. 
Consider the two-dimensional elastic wave problem of finding
$\boldsymbol{u}$ with zero trace on the boundary of $D$, i.e., $\Gamma$, such that
\begin{equation}\label{ElasticityLHS1}
 \nabla\cdot\sigma(\boldsymbol{u})+\omega^2 \rho \boldsymbol{u}={\boldsymbol 0}
\quad\text{in}~D \subset \mathbb{R}^2,
\end{equation}
where $\boldsymbol{u}(\boldsymbol{x})=(u_1(\boldsymbol{x}),
u_2(\boldsymbol{x}))^\top$ is the displacement vector of the wave field,
$\omega>0$ is the angular frequency, $\rho$ is the mass density, and
$\sigma(\boldsymbol{u})$ is  the stress tensor given by the generalized Hooke law
\[
 \sigma(\boldsymbol{u})=2\mu\varepsilon(\boldsymbol{u})+\lambda {\rm
tr}(\varepsilon(\boldsymbol{u})){\rm I}.
\]
Here the strain tensor $\varepsilon(\boldsymbol{u})$ is given by  
\[
\varepsilon(\boldsymbol{u})=\frac{1}{2}(\nabla\boldsymbol{u}
+(\nabla\boldsymbol{u})^\top),
\]
where the two constants $\mu, \lambda$ are called the Lam\'{e} parameters satisfying
$\mu>0, \lambda+\mu>0$, ${\rm I}\in\mathbb{R}^{2\times 2}$ is the identity
matrix, and $\nabla\boldsymbol{u}$ is the displacement gradient tensor
\[
 \nabla\boldsymbol{u}=\begin{bmatrix}
                       \partial_x u_1 & \partial_y u_1\\
                       \partial_x u_2 & \partial_y u_2
                      \end{bmatrix}.
\]
Explicitly, we have
\begin{equation}\label{DefSigma}
 \sigma(\boldsymbol{u})=\begin{bmatrix}
  (\lambda+2\mu)\partial_x u_1 + \lambda \partial_y u_2 & \mu (\partial_y u_1 +
\partial_x u_2)\\
\mu (\partial_x u_2 + \partial_y u_1) & \lambda\partial_x u_1 +
(\lambda+2\mu)\partial_y u_2
 \end{bmatrix}.
\end{equation}

Given
${\boldsymbol u}, {\boldsymbol v} \in H_0^1(D)^2$, it follows from the
integration by parts that 
\begin{equation}\label{sigmaugv}
 ({\sigma} ({\boldsymbol u}), \nabla {\boldsymbol v})=\int_D
\sigma(\boldsymbol u):\nabla\boldsymbol v \,{\rm d}\boldsymbol x=
\int_D \left( 2 \mu \varepsilon({\boldsymbol u}): 
\varepsilon({\boldsymbol v}) +\lambda (\nabla\cdot{\boldsymbol u})( 
\nabla\cdot{\boldsymbol v})  \right){\rm d}{\boldsymbol x},
\end{equation}
where $A:B={\rm tr}(AB^\top)$ is the Frobenius inner product of square
matrices $A$ and $B$. We recall the first Korn inequality \cite[Corollary
11.2.25]{BrennerScott2002}: there exists a positive constant $C$ such that
\[
\|\varepsilon({\boldsymbol u})\|_{L^2} \ge C \|{\boldsymbol
u}\|_{H^1} \quad \text{for all } {\boldsymbol u} \in H_0^1(D)^2,
\]
which guarantees the well-posedness of \eqref{sigmaugv}.


Let $\mu_0, \lambda_0$ be the Lam\'{e} parameters of the free space.
Assume the domain $D$ is filled with a
homogeneous and isotropic elastic medium with Lam\'{e}
constants $\lambda_1$ and $\mu_1$. The transmission eigenvalue problem for the
elastic waves is to find values of $\omega^2$ such that there exists non-trivial
solutions $\boldsymbol{u}, \boldsymbol{v}$ satisfying
\begin{subequations}\label{tep}
\begin{align}
\nabla\cdot\sigma_0(\boldsymbol{u})+\omega^2 \rho_0 \boldsymbol{u}=0 &\quad\text{in} ~
D,\\
\nabla\cdot\sigma_1(\boldsymbol{v})+\omega^2 \rho_1 \boldsymbol{v}=0 &\quad\text{in} ~
D,\\
\boldsymbol{u}=\boldsymbol{v} &\quad\text{on} ~ \Gamma,\\
\sigma_0(\boldsymbol{u}){\boldsymbol \nu}=\sigma_1(\boldsymbol{v}){\boldsymbol \nu} &\quad\text{on} ~
\Gamma,
\end{align}
\end{subequations}
where
\[
 \sigma_i(\boldsymbol{u})=\begin{bmatrix}
  (\lambda_i+2\mu_i)\partial_x u_1 + \lambda_i \partial_y u_2 & \mu_i
(\partial_y u_1 + \partial_x u_2)\\
\mu_i (\partial_x u_2 + \partial_y u_1) & \lambda_i\partial_x u_1 +
(\lambda_i+2\mu_i)\partial_y u_2
 \end{bmatrix},\quad i=0,1,
\]
and $\sigma {\boldsymbol \nu}$ denotes the matrix multiplication of the
stress tensor $\sigma$ and the normal vector $\boldsymbol \nu$.

In this paper, we consider the case when $\rho_0 \ne \rho_1, \sigma_0 =
\sigma_1=\sigma$, i.e., the case of equal elastic
tensors \cite{BellisCakoniGuzina2013IMAAM}. In addition, we assume that the mass
density distributions satisfy the following conditions
\begin{equation}\label{pP}
p \le \rho_0({\boldsymbol x}) \le P, \quad p_* \le 
\rho_1({\boldsymbol x}) \le P_*, \quad {\boldsymbol x} \in D,
\end{equation}
where $p, p_*$ and $P, P_*$ are positive constants. 

Define the Sobolev space
\begin{equation}\label{spaceV}
V=\{ \boldsymbol{\phi} \in H^2(D)^2: {\boldsymbol \phi} = {\boldsymbol 0} \text{
and } \sigma({\boldsymbol \phi}) {\boldsymbol \nu} = {\bf 0} ~ \text{on}~ \Gamma
\}.
\end{equation}
Let ${\boldsymbol w} = {\boldsymbol  u} - {\boldsymbol  v}$. The transmission
eigenvalue problem can be formulated as follows: Find $\omega^2$ and $ 
{\boldsymbol w} \ne {\bf 0}$ such that
\begin{equation}\label{FourthOrder}
\left(\nabla \cdot {\boldsymbol \sigma} + \omega^2 \rho_1 \right)
(\rho_1-\rho_0)^{-1} \left(\nabla \cdot {\boldsymbol \sigma} + \omega^2 \rho_0
\right) {\boldsymbol w} = {\bf 0}.
\end{equation}
The corresponding weak formulation of \eqref{FourthOrder} is to find $\omega^2
\in \mathbb C$ and ${\bf 0} \ne {\boldsymbol  w} \in V$ such that
\begin{equation}\label{FourthOrderWF}
 \left( (\rho_1-\rho_0)^{-1} \left(\nabla \cdot {\sigma} + \omega^2 \rho_0
\right) {\boldsymbol w},
 \left(\nabla \cdot {\sigma} + \omega^2 \rho_1 \right){\boldsymbol
\varphi}\right)= {\bf 0}\quad \text{for all } {\boldsymbol \varphi} \in V.
\end{equation}
Let $\tau= \omega^2$. We define two sesquilinear forms on $V \times V$
\begin{align*}
\mathcal{A}_\tau({\boldsymbol \phi}, {\boldsymbol \varphi}) &= 
\left( (\rho_1-\rho_0)^{-1} \left(\nabla \cdot {\sigma} + \tau \rho_0 \right)
{\boldsymbol \phi}, \left(\nabla \cdot {\sigma} + \tau \rho_0
\right){\boldsymbol \varphi}\right) + \tau^2 (\rho_0 {\boldsymbol \phi},
{\boldsymbol \varphi}),\\
\mathcal{B}({\boldsymbol \phi}, {\boldsymbol \varphi})  &= ({\sigma}
({\boldsymbol \phi}), \nabla {\boldsymbol \varphi}).
\end{align*}
It is clear that $\mathcal{A}_\tau$ is symmetric. Due to \eqref{sigmaugv},
$\mathcal{B}$ is also symmetric.

The variational problem \eqref{FourthOrderWF} can be written equivalently as
follows: Find $\tau \in \mathbb C$ and ${\bf 0} \ne {\boldsymbol 
w} \in V$ such that
\begin{equation}\label{Nonlinear}
\mathcal{A}_\tau({\boldsymbol w}, {\boldsymbol \varphi}) = \tau
\mathcal{B}({\boldsymbol w}, {\boldsymbol \varphi})  \quad \text{for all }
{\boldsymbol \varphi} \in V.
\end{equation}
This is a nonlinear problem since $\tau$ appears on both sides of the equation.
For a fixed $\tau$, we consider an associated generalized eigenvalue problem
\begin{equation}\label{Linear}
\mathcal{A}_\tau({\boldsymbol w}, {\boldsymbol \varphi}) = \gamma(\tau)
\mathcal{B}({\boldsymbol w}, {\boldsymbol \varphi}) \quad \text{for all }
{\boldsymbol \varphi} \in V.
\end{equation}
Formally, $\tau$ is a transmission eigenvalue if
$\tau$ is a root of the nonlinear function
\begin{equation}\label{functionf}
f(\tau):=\gamma(\tau) - \tau.
\end{equation}

In the rest of this section, we study the generalized eigenvalue problem
\eqref{Linear}. It is shown in \cite{MarsdenHughes1994} that there exists
$\beta>0$ such that
\[
\|\nabla \cdot {\sigma}({\boldsymbol \phi})\|^2 +  \| {\boldsymbol \phi}\|^2 \ge
\beta \|{\boldsymbol \phi}\|^2_{H^2(D)^2} \quad \text{for } {\boldsymbol
\phi} \in V.
\]

The following lemma is useful in the subsequent analysis. The proof can be
found in \cite{BellisCakoniGuzina2013IMAAM}.

\begin{lemma}\label{BellisLemma}
Assume that $p_* \ge 1 \ge P$. Then $\mathcal{A}_\tau$ is a coercive
sesquilinear form on $V \times V$, i.e., there exists a constant $\alpha > 0$
such that
\[
\mathcal{A}_\tau({\boldsymbol \phi}, {\boldsymbol \phi}) \ge \alpha
\|{\boldsymbol \phi}\|^2 \quad \text{for all } {\boldsymbol \phi} \in V.
\]
\end{lemma}

The source problem associated with \eqref{Linear} is to find ${\boldsymbol u}
\in V$ such that, for ${\boldsymbol f} \in H^1(D)^2$,
\begin{equation}\label{source}
\mathcal{A}_\tau({\boldsymbol u}, {\boldsymbol \phi}) = (\sigma({\boldsymbol
f}), \nabla {\boldsymbol \phi})\quad\text{for all } {\boldsymbol \phi} \in V.
\end{equation}
The following theorem is a direct consequence of the Lax--Milgram Lemma.
\begin{theorem}\label{UEsource}
There exists a unique solution ${\boldsymbol u} \in V$ to \eqref{source}. Furthermore, it holds that
\begin{equation}\label{uH2norm}
\|{\boldsymbol u}\|_{H^2(D)^2} \le C \|{\boldsymbol f}\|_{H^1(D)^2}.
\end{equation}
\end{theorem}
\begin{proof} It is easy to show that $\mathcal{A}_\tau$ is bounded. The
coercivity of $\mathcal{A}_\tau$ follows Lemma \ref{BellisLemma}. Let $F$ be a
linear functional on $V$ such that
\[
F({\boldsymbol \phi}):=(\sigma({\boldsymbol f}), \nabla {\boldsymbol \phi}),
\]
for all ${\boldsymbol \phi} \in V$. Then the Lax-Milgram Lemma implies that
there exists a unique solution ${\boldsymbol u}$ to the problem
\[
\mathcal{A}_\tau({\boldsymbol u}, {\boldsymbol \phi}) = F({\boldsymbol \phi})
\quad \text{for all } {\boldsymbol \phi} \in V.
\]
Moreover, we have
\[
\|{\boldsymbol u}\|_{H^2(D)^2} \le C \|F\|_{V'},
\]
where $V'$ represents the dual space of $V$. Following from the definition of
$\sigma({\boldsymbol f})$, we obtain from a simple calculation that
\[
\|F\|_{V'} \le C_{\lambda, \mu} \|{\boldsymbol f}\|_{H^1(D)^2},
\]
which shows the estimate \eqref{uH2norm} and completes the proof. 
\end{proof}

\begin{rema}
In the rest of the paper, we assume that the following regularity for ${\boldsymbol u}$ holds
\begin{equation}\label{FourthReg}
  \|{\boldsymbol u}\|_{H^{2+\xi}(\Omega)^2}\leq C \|{\boldsymbol
f}\|_{H^1(D)^2}.
\end{equation}
Note that a similar regularity holds for the biharmonic equation
\cite{BR1980,Grisvard1985, BrennerMonkSun2015LNCSE}
where the elliptic regularity $\xi \in (\frac{1}{2},1]$ is determined by the
angles at the corners of $D$ and $\xi=1$ if $D$ is convex.
 \end{rema}

It follows from Theorem \ref{UEsource} that there exists a solution operator
$T: H^1(D)^2 \to V$ such that
 \[
 {\boldsymbol u} = T{\boldsymbol f}.
 \]
Clearly, the operator $T$ is self-adjoint since $\mathcal{A}_\tau$ is
symmetric; $T$ is also a compact operator due to the compact embedding of
$H^2(D)^2$ into $H^1(D)^2$ (see, e.g., Theorem 1.2.1 of \cite{SunZhou2016}).
The generalized eigenvalue problem \eqref{Linear} has the following equivalent
operator form
\[
{\boldsymbol u} = \eta T{\boldsymbol u},\quad \text{where } \eta= \gamma^{-1}.
\]

From classical spectral theory of compact self-adjoint operators, i.e., the
Hilbert-Schmidt theory, $T$ has at most a countable set of real eigenvalues and
$0$ is the only possible accumulation point. Consequently, we have the following
lemma for the generalized eigenvalue value problem \eqref{Linear}.

\begin{lemma}
Let $\rho_0$ and $\rho_1$ satisfy \eqref{pP} such that the condition in Lemma \eqref{BellisLemma} is fulfilled. 
Then the generalized eigenvalue value problem \eqref{Linear} has at most a countable set of positive eigenvalues 
and $+\infty$ is the only possible accumulation point.
\end{lemma}

Roughly speaking, to compute real  transmission eigenvalues, one needs to computes the
roots of the nonlinear function $f$. The values of $f(\tau)$ are
generalized eigenvalues of \eqref{Linear}, which is approximated by the $H^2$-conforming Argyris element.

\section{A conforming finite element method}

In this section, we propose a conforming finite element for \eqref{Linear}. 
The convergence of the source problem \eqref{source} is established first. The theory of Babu\v{s}ka and Osborn
\cite{BabuskaOsborn1991} is then applied to obtain the convergence of the eigenvalue problem \eqref{Linear}.

Let $\mathcal{T}$ be a regular triangular mesh for $D$ and $K \in \mathcal{T}$ be a triangle. We employ the
$H^2$-conforming Argyris element, which uses $\mathcal{P}_5$ - the set of polynomials of
degree up to $5$ on $K$, to discretize \eqref{Linear}. Note that $\text{dim}(\mathcal{P}_5) = 21$. For $\mathcal{N}=\{N_1,
\ldots, N_{21}\}$, $3$ degrees of freedom are the values at the vertices of
$K$, $6$ degrees of freedom are the values of the first order partial
derivatives at the vertices of $K$, $9$ degrees of freedoms are the values of
the second order derivatives at the vertices of $K$,  and $3$ degrees of freedom
are the values of the normal derivatives at the midpoints
of three edges of $K$ \cite{BrennerScott2002}.

Note that the Argyris element does not belong to the affine families. This is
due to the fact that normal derivatives are used as degrees of
freedom. Fortunately, their interpolation properties are quite similar to those
of affine families. Hence the Argyris element is referred to be almost-affine
element. Let $v \in H^2(D)$ and $I_hv$ be the interpolation of $v$ by the
Argyris element. For $v \in H^{1+\alpha}(D)$, $\alpha > 0$, the following interpolation result holds (see, e.g., \cite{Ciarlet2002})
\begin{equation}\label{ArgyrisInter}
\| v - I_h v\|_{H^2(D)} \le C h^{s-1}|v|_{H^{s+1}(D)},
\end{equation}
where $1 \le s \le \min\{5, 1+\alpha\}$ depending regularity of $v$.

Let $V_h$ be the Argyris finite element space associated with $\mathcal T$. The
discrete problem for \eqref{source} is to find ${\boldsymbol u}_h \in V_h$ such that
\begin{equation}\label{SourceDis}
\mathcal{A}_\tau({\boldsymbol u}_h, {\boldsymbol \phi}_h) = (\sigma({\boldsymbol
f}), \nabla {\boldsymbol \phi}_h)\quad \text{for all } {\boldsymbol \phi}_h \in
V_h.
\end{equation}
The existence of a unique solution ${\boldsymbol u}_h$ to \eqref{SourceDis} is
the same as the continuous problem. As a consequence, there exists a discrete
solution operator
$T_h: H^1(D)^2 \to H^2(D)^2$ such that
\[
 {\boldsymbol u}_h = T_h{\boldsymbol f}.
\]

\begin{theorem}
Let ${\boldsymbol u}$ and ${\boldsymbol u}_h$ be the solutions of the continuous
problem \eqref{source} and discrete problem \eqref{SourceDis}, respectively.
Then the following error estimate holds
\[
\|{\boldsymbol u}-{\boldsymbol u}_h\|_{H^1(D)^2} \le C h^{2\alpha} \|{\boldsymbol f}\|_{H^1(D)^2}.
\]
\end{theorem}

\begin{proof}
From C\'{e}a's Lemma, the following error estimate holds
\[
\|{\boldsymbol u}-{\boldsymbol u}_h\|_{H^2(D)^2} \le C \inf_{{\boldsymbol v}_h \in V_h} \|{\boldsymbol u}-{\boldsymbol v}_h\|_{H^2(D)^2},
\]
for some constant $C$. Using \eqref{ArgyrisInter} and \eqref{FourthReg}, one has that
\[
\|{\boldsymbol u}-{\boldsymbol u}_h\|_{H^2(D)^2} \le C h^{s-1}|{\boldsymbol u}|_{H^{s+1}(D)^2}
= Ch^\alpha |{\boldsymbol u}|_{H^{2+\alpha}(D)^2} \le C h^\alpha \|{\boldsymbol f}\|_{H^1(D)^2}.
\]
For ${\boldsymbol g} \in H_0^1(D)^2$, let ${\boldsymbol \phi}_{\boldsymbol g}$ be the unique solution of
\[
\mathcal{A}_\tau({\boldsymbol \phi}_g, {\boldsymbol \phi}) = (\sigma({\boldsymbol g}), \nabla {\boldsymbol \phi}) \quad \text{for all } {\boldsymbol \phi}  \in V.
\]
The rest of the proof follows the Aubin--Nitsche Lemma (see, e.g., Theorem 3.2.4
of \cite{SunZhou2016}) with suitable choices of Sobolev spaces.
Let ${\boldsymbol w} := {\boldsymbol u}-{\boldsymbol u}_h$ and ${\boldsymbol g}
\in H^1(D)^2$. Using the Galerkin orthogonality, we have for any ${\boldsymbol
v}_h \in V_h$ that
\begin{align*}
(\sigma({\boldsymbol g}), \nabla {\boldsymbol w}) &=
\mathcal{A}_\tau({\boldsymbol \phi}_g,  {\boldsymbol u}-{\boldsymbol u}_h) \\
			&=  \mathcal{A}_\tau({\boldsymbol \phi}_g - {\boldsymbol
v}_h,  {\boldsymbol u}-{\boldsymbol u}_h) \\
			&\le C \| {\boldsymbol \phi}_g - {\boldsymbol v}_h
\|_{H^2} \| {\boldsymbol u}-{\boldsymbol u}_h\|_{H^2},
\end{align*}
which yields 
\[
(\sigma({\boldsymbol g}), \nabla {\boldsymbol w}) \le C  \| {\boldsymbol u}-{\boldsymbol u}_h\|_{H^2} \inf_{{\boldsymbol v}_h \in V_h} \| {\boldsymbol \phi}_g - {\boldsymbol v}_h \|_{H^2}.
\]
Furthermore,
\begin{eqnarray*}
\| {\boldsymbol u}-{\boldsymbol u}_h\|_{H^1} &=& \sup_{{\boldsymbol g} \in H^1(D)^2, {\boldsymbol g} \ne {\boldsymbol 0}}
\frac{( {\boldsymbol u}-{\boldsymbol u}_h, {\boldsymbol g})}{\|{\boldsymbol g}\|_{H^1}} \\
&\le&  C  \| {\boldsymbol u}-{\boldsymbol u}_h\|_{H^2}
 \sup_{{\boldsymbol g} \in H^1(D)^2, {\boldsymbol g} \ne {\boldsymbol 0}}
 \left\{ \inf_{{\boldsymbol v}_h \in V_h} \frac{\| {\boldsymbol \phi}_g - {\boldsymbol v}_h \|}{\|{\boldsymbol g}\|_{H^1}}\right\}.
\end{eqnarray*}
Consequently, we get
\[
\|{\boldsymbol u}-{\boldsymbol u}_h\|_{H^1(D)^2} \le C h^{2\alpha} \|{\boldsymbol f}\|_{H^1(D)^2}.
\]
which completes the proof.
\end{proof}

Using operators $T$ and $T_h$, we can rewrite the above error estimate as
 \[
 \|T{\boldsymbol f}-T_h{\boldsymbol f} \| \le C h^{2\alpha} \|{\boldsymbol f}\|_{H^1(D)^2}.
 \]
Thus we have
\[
\|T-T_h\| \le Ch^{2\alpha}.
\]

Now we consider the discrete eigenvalue value problem: Find $\gamma_h \in
\mathbb R$ such that
\begin{equation}\label{EigDis}
\mathcal{A}_\tau({\boldsymbol u}_h, {\boldsymbol \phi}_h) =
\gamma_h(\sigma({\boldsymbol u}_h), \nabla {\boldsymbol \phi}_h) \quad \text{for
all } {\boldsymbol \phi}_h \in V_h.
\end{equation}

Since both $\mathcal{A}_\tau$ and $\mathcal{B}$ are symmetric, $T$ is
self-adjoint. Similarly, $T_h$ is symmetric.
The estimate of eigenvalue problem follows directly from the theory of
Babu\v{s}ka and Osborn \cite{BabuskaOsborn1991}.

\begin{theorem}\label{convergeorder}
Let $\gamma$ be a generalized eigenvalue of \eqref{Linear} with algebraic
multiplicity $m$. Let $\gamma_{h, 1}, \ldots, \gamma_{h, m}$ be the $m$
eigenvalues of \eqref{EigDis} approximating $\mu$. Define $\hat{\gamma}_h =
\frac{1}{m}\sum_{j=1}^m \gamma_{h, j}$. The following estimate holds
\[
|\gamma - \hat{\gamma}_h| \le Ch^{2\alpha},
\]
where $C>0$ is a constant. 
\end{theorem}

\begin{rema}
The boundary conditions for $V$ defined in \eqref{spaceV} need special treatment. The detail of how to impose the boundary conditions is shown in Appendix~B.
\end{rema}

\section{An iterative method}

Now we turn to the problem of how to compute the root(s) of
the nonlinear function $f(\tau)$ defined in \eqref{functionf}. In this section, we assume that $\rho_0$ and $\rho_1$ are constants
and consider the case when $\gamma(\tau)$ is the smallest eigenvalue of \eqref{Linear}. Similar result holds for other eigenvalues.
The continuity of $f$ is clear since the generalized eigenvalue $\gamma(\tau)$ of \eqref{Linear} depends on $\tau$ continuously.
The following lemma is shown in \cite{BellisCakoniGuzina2013IMAAM}. It is written in a slightly different way to better serve the current paper.
\begin{lemma}\label{fcontinuous}
Let $\tau_0 > 0$ be small enough and $\tau_1>0$ be large enough. Then the nonlinear function $f$ is continuous and has at least one root in $[\tau_0, \tau_1]$.
\end{lemma}

In fact, $f$ is differentiable and the derivative is negative on an interval given in Theorem~\ref{decreasingfunc}.
We first recall the elasticity eigenvalue problem which will be used in the proof (see, e.g., \cite{BabuskaOsborn1991}). 
Find non-trivial eigenpair $(\delta, {\boldsymbol u}) \in \mathbb R \times \in H_0^1(D)^2$
such that
\begin{equation}\label{EEweak}
\int_D \left( 2\mu {\boldsymbol \epsilon}({\boldsymbol u}):{\boldsymbol \epsilon}({\boldsymbol u})+
\lambda \text{div}{\boldsymbol u} \text{div}{\boldsymbol v} \right) \text{d}{\boldsymbol x} = \delta \int_D {\boldsymbol u} {\boldsymbol v} \text{d}{\boldsymbol x}
\quad \text{for all }  {\boldsymbol v} \in H_0^1(D)^2.
\end{equation}

\begin{theorem}\label{decreasingfunc}
Let $\delta_1$ be the smallest elasticity eigenvalue.
The function $f(\tau)$ is differentiable. Furthermore,
$f(\tau)$ is a decreasing function on $\left(0, 
\frac{\delta_1(\rho_0+\rho_1)}{2\rho_0 \rho_1}\right)$.
\end{theorem}

\begin{proof}
Let $\gamma_1(\tau, \rho_0, \rho_1)$ be the first generalized eigenvalue of
\eqref{Linear}. The following Rayleigh quotient holds
\begin{eqnarray*}
\gamma_1(\tau, \rho_0, \rho_1) &=& \inf_{{\boldsymbol w} \in V} \frac{\mathcal{A}_\tau({\boldsymbol w}, {\boldsymbol w})}{\mathcal{B}({\boldsymbol w}, {\boldsymbol w}) } \\
&=& \inf_{{\boldsymbol w} \in V} \frac{  \left(\frac{1}{\rho_1-\rho_0} \left(\nabla \cdot {\sigma} + \tau \rho_0 \right) {\boldsymbol w},
 \left(\nabla \cdot {\sigma} + \tau \rho_0 \right){\boldsymbol w}\right) + \tau^2 (\rho_0 {\boldsymbol w}, {\boldsymbol w})}
 {({\sigma} ({\boldsymbol w}), \nabla {\boldsymbol w})} \\
 &=& \inf_{{\boldsymbol w} \in V} \frac{  \left(\frac{1}{\rho_1-\rho_0} \nabla \cdot {\sigma}({\boldsymbol w}), \nabla \cdot {\sigma}({\boldsymbol w})\right)
 	+ 2\tau \left(\frac{\rho_0}{\rho_1-\rho_0}  {\boldsymbol w}, \nabla \cdot {\sigma}({\boldsymbol w})\right)
+ \tau^2 \left(\frac{\rho_0 \rho_1}{\rho_1-\rho_0} {\boldsymbol w}, {\boldsymbol w} \right)}
 {({\sigma} ({\boldsymbol w}), \nabla {\boldsymbol w})}.
\end{eqnarray*}
When $\rho_0$ and $\rho_1$ are constants, we have 
\[
\gamma_1(\tau, \rho_0, \rho_1) =\inf_{{\boldsymbol w} \in V} \frac{  \frac{1}{\rho_1-\rho_0} \left(\nabla \cdot {\sigma}({\boldsymbol w}), \nabla \cdot {\sigma}({\boldsymbol w})\right)
+ \tau^2 \frac{\rho_0 \rho_1}{\rho_1-\rho_0} \left({\boldsymbol w}, {\boldsymbol w} \right)}
 {({\sigma} ({\boldsymbol w}), \nabla {\boldsymbol w})} - \frac{2\tau \rho_0}{\rho_1-\rho_0}.
\]

Note that the sesquilinear form 
\[
a({\boldsymbol u}, {\boldsymbol v}):=({\sigma} ({\boldsymbol u}), \nabla
{\boldsymbol v}) = 2\mu \varepsilon({\boldsymbol u}):\varepsilon({\boldsymbol
v}) +\lambda (\nabla\cdot{\boldsymbol u}) (\nabla\cdot{\boldsymbol v})
\]
is bounded, symmetric, and coercive. Hence
\[
\gamma_1(\tau, \rho_0, \rho_1) =\inf_{{\boldsymbol w} \in V, a({\boldsymbol w}, {\boldsymbol w})=1} \left\{\frac{1}{\rho_1-\rho_0} \left(\nabla \cdot {\sigma}({\boldsymbol w}), \nabla \cdot {\sigma}({\boldsymbol w})\right)
+ \tau^2 \frac{\rho_0 \rho_1}{\rho_1-\rho_0} \left({\boldsymbol w}, {\boldsymbol w} \right) \right\}
 - \frac{2\tau \rho_0}{\rho_1-\rho_0}.
\]
Let $\kappa :=\tau^2 \frac{\rho_0 \rho_1}{\rho_1-\rho_0}$. We define a new function
\[
s(\kappa) = \inf_{{\boldsymbol w} \in V, a({\boldsymbol w}, {\boldsymbol w})=1} \left\{\frac{1}{\rho_1-\rho_0} \|\nabla \cdot {\sigma}({\boldsymbol w})\|^2
+\kappa\|{\boldsymbol w}\|^2 \right\}.
\]
For a fixed $\kappa \in (0, \infty)$, there exists a ${\boldsymbol w}_\kappa$ such that ${\boldsymbol w_\kappa} \in V, a({\boldsymbol w}_\kappa, {\boldsymbol w}_\kappa)=1$, and
\[
s(\kappa) =  \left\{\frac{1}{\rho_1-\rho_0} \|\nabla \cdot {\sigma}({\boldsymbol w}_\kappa)\|^2
+\kappa\|{\boldsymbol w}_\kappa\|^2 \right\}.
\]
For a small enough positive $h$,
\begin{eqnarray*}
s(\kappa+h)-s(\kappa) &\le& \left\{\frac{1}{\rho_1-\rho_0} \|\nabla \cdot {\sigma}({\boldsymbol w}_\kappa)\|^2
+(\kappa+h)\|{\boldsymbol w}_\kappa\|^2 \right\} - \left\{\frac{1}{\rho_1-\rho_0} \|\nabla \cdot {\sigma}({\boldsymbol w}_\kappa)\|^2+\kappa\|{\boldsymbol w}_\kappa\|^2 \right\} \\
&=& h \|{\boldsymbol w}_\kappa\|^2.
\end{eqnarray*}
On the other hand, we have 
\begin{eqnarray*}
s(\kappa+h)-s(\kappa) &\ge& \left\{\frac{1}{\rho_1-\rho_0} \|\nabla \cdot {\sigma}({\boldsymbol w}_{\kappa+h})\|^2
+(\kappa+h)\|{\boldsymbol w}_{\kappa+h}\|^2 \right\} \\
&& \qquad \qquad - \left\{\frac{1}{\rho_1-\rho_0} \|\nabla \cdot {\sigma}({\boldsymbol w}_{\kappa+h})\|^2+
\kappa\|{\boldsymbol w}_{\kappa+h}\|^2 \right\} \\
&=& h \|{\boldsymbol w}_{\kappa+h}\|^2.
\end{eqnarray*}
Consequently,
\[
\|{\boldsymbol w}_{\kappa+h}\|^2 \le \frac{s(\kappa+h)-s(\kappa)}{h} \le \|{\boldsymbol w}_\kappa\|^2.
\]

The above inequality implies that  $\|{\boldsymbol w}_\kappa\|^2$ is monotonically decreasing and thus bounded.
Note that $a({\boldsymbol w}_\kappa, {\boldsymbol w}_\kappa)=1$. Then the continuity of $s$ and the compact
embedding of $V$ into $L^2(D)^2$ imply the existences of a
$\tilde{\boldsymbol w}$ such that
$\boldsymbol w_{\kappa+h}$ converges in $L^2(D)^2$ strongly and $\boldsymbol w_{\kappa+h}$ converges in $H^2(D)^2$ weakly.
In addition, ${\boldsymbol w}_{\kappa+h}$ satisfies
\[
\left( \frac{1}{\rho_1-\rho_0}\nabla \cdot {\sigma}({\boldsymbol w}_{\kappa+h}), \nabla \cdot {\sigma}({\boldsymbol \phi})\right)
+(\kappa+h)\left( {\boldsymbol w}_{\kappa+h},  {\boldsymbol \phi} \right) =  s(k+h) \left( {\sigma} ({\boldsymbol w}), \nabla {\boldsymbol \phi}\right),
\]
for all ${\boldsymbol \phi} \in V$. Taking $h \to 0$, we obtain
\[
\left( \frac{1}{\rho_1-\rho_0}\nabla \cdot {\sigma}(\tilde{\boldsymbol w}), \nabla \cdot {\sigma}({\boldsymbol \phi})\right)
+\kappa \left( \tilde{\boldsymbol w},  {\boldsymbol \phi} \right) =  s(k) \left( {\sigma} (\tilde{\boldsymbol w}), \nabla {\boldsymbol \phi}\right),
\]
for all ${\boldsymbol \phi} \in V$. Thus $\tilde{\boldsymbol w} = {\boldsymbol w}_\kappa$. Consequently
\[
\|{\boldsymbol w}_{\kappa+h}\|^2 \to \|{\boldsymbol w}_{\kappa}\|^2, \qquad h \to 0.
\]
Then the derivative of $s(\kappa)$ is $\|{\boldsymbol w}_{\kappa}\|^2$.

Combing the above estimates, we obtain 
\begin{eqnarray*}
\frac{\partial f(\tau)} {\partial \tau} &=& 2 \tau \frac{\rho_0 \rho_1}{\rho_1 - \rho_0}\|{\boldsymbol w}_{\kappa}\|^2 - \frac{2\rho_0}{\rho_1-\rho_0}-1\\
&=&  2 \tau \frac{\rho_0 \rho_1}{\rho_1 - \rho_0}\|{\boldsymbol w}_{\kappa}\|^2 - \frac{\rho_1+\rho_0}{\rho_1-\rho_0}.
\end{eqnarray*}
Let $\delta_1$ be the smallest elasticity eigenvalue. One has that
\[
\|{\boldsymbol w}_\kappa\|^2 \le \frac{1}{\delta_1} ({\sigma} ({\boldsymbol w}_\kappa), \nabla {\boldsymbol w}_\kappa) = \frac{1}{\delta_1},
\]
since $ (\sigma({\boldsymbol w}_\kappa), \nabla {\boldsymbol w}_\kappa) = 1$.
This implies that
\begin{equation}\label{pfptau}
\frac{\partial f(\tau)} {\partial \tau} \le  \frac{ 2 \tau \rho_0 \rho_1}{\delta_1(\rho_1 - \rho_0)} - \frac{\rho_1+\rho_0}{\rho_1-\rho_0}.
\end{equation}
In particular, $f$ is decreasing, i.e.,
\[
\frac{\partial f(\tau)} {\partial \tau} \le 0 \qquad \text{if } \qquad \tau <
\frac{\delta_1(\rho_0+\rho_1)}{2\rho_0 \rho_1}.
\]
It is easy to see that $f(\tau)>0$ if $\tau \to 0$ and $f(\tau) < 0$ if $\tau \to \infty$.
\end{proof}

Since we only have the finite element approximation for the values for $f$, the nonlinear
equation which we solve is in fact a discrete version of \eqref{functionf}
\begin{equation}\label{functionfh}
f_h(\tau):=\gamma_h(\tau) - \tau.
\end{equation}
Let $\epsilon > 0$. From \eqref{pfptau}, there exists $\eta>0$ such that  
\begin{equation}\label{pftau}
\frac{\partial f(\tau)} {\partial \tau} \le - \eta \qquad \text{for } \tau \in \left [0, \frac{\delta_1(\rho_0+\rho_1)}{2\rho_0 \rho_1} - \epsilon \right].
\end{equation}
\begin{theorem}\label{convergenceA}
Assume that we apply the conforming Argyris finite element method for \eqref{Linear} on a regular mesh $\mathcal T$ with mesh size $h$.
Let $\tau_0$ be the exact root of \eqref{functionf} and $\tau_{0,h}$ be the root of \eqref{functionfh} such that 
$\tau_0, \tau_{o,h} \in \left [0, \frac{\delta_1(\rho_0+\rho_1)}{2\rho_0 \rho_1} - \epsilon \right]$.
Then there exists $h_0$ such that for $h < h_0$ 
\begin{equation}\label{gamma10}
|\tau_{0,h}-\tau_0| \le Ch^{2\alpha}/\eta.
\end{equation}
\end{theorem}
\begin{proof}
The assumption implies that $\gamma(\tau_0)-\tau_0$ and $\gamma_h(\tau_{0,h}) = \tau_{0,h}$, i.e.,
\[
\gamma = \tau_0 \quad \text{and} \quad \gamma_h = \tau_{0,h}.
\]
By Theorem \ref{convergeorder}, there exist $h_0$ such that for a regular mesh with $h < h_0$, we have
\[
|\tau_{0,h}-\tau_0| = |\gamma - \gamma_h| < C h^{2\alpha}.
\]
Then direct application of \eqref{pftau} leads to \eqref{gamma10}.
\end{proof}

Note that $f_h(\tau)$ is a nonlinear function. It is natural to use some iterative methods to compute the roots of $f_h$. We choose to
use the secant method which avoids using the derivatives of $f_h(\tau)$, 

Given a regular triangular mesh $\mathcal{T}$ for $D$, let $x_0$ and $x_1$ be
two positive numbers close to $0$ such that $0 < x_0 < x_1$. Let $N$ be the
number of smallest real transmission eigenvalues one wants to compute. Let
$tol$ and $maxit$ be the preset precision and the maximum number
of iteration of the secant method, respectively. 
The following algorithm uses a secant iteration to compute $N$ smallest positive transmission eigenvalues.

\vskip 0.2cm
{\bf $\text{SMETE}$}
\begin{itemize}
\item[] construct matrix $B_{h}$ corresponding to $\mathcal{B}$ in \eqref{Linear}
\item[] for $i=1 : N$ 
\item[] $\quad$ $it=0$
\item[] $\quad$ $\delta=\text{abs}(x_1-x_0)$
\item[] $\quad$ $\tau= x_0$ and construct the matrix $A_{\tau,h}$ 
\item[] $\quad$ $\mbox{compute the $i$th generalized eigenvalue $\gamma_0$ of $A_{\tau,h}{\bf x}=\gamma B_h{\bf x}$}$
\item[] $\quad$ $\tau = x_1$ and construct matrix $A_{\tau,h}$
\item[] $\quad$ $\mbox{compute the $i$th generalized eigenvalue $\gamma_1$ of $A_{\tau,h}{\bf x}=\gamma B_h{\bf x}$}$
\item[] $\quad$ $\mbox{while $\delta > tol$ and $it< maxit$}$
\item[] $\quad$ $\qquad \tau = x_1-\gamma_1 \frac{x_1-x_0}{\gamma_1-\gamma_0}$
\item[] $\quad$ $\qquad \mbox{construct the matrix $A_{\tau,h}$}$
\item[] $\quad$ $\qquad \mbox{compute the $i$th smallest eigenvalue $\gamma_\tau$ of $A_{\tau,h}{\bf x}=\gamma B_h{\bf x}$}$
\item[] $\quad$ $\qquad \delta = \text{abs}(\gamma_\tau-\tau)$
\item[] $\quad$ $\qquad x_0 = x_1, x_1=\tau$
\item[] $\quad$ $\qquad \gamma_0=\gamma_1, \gamma_1=\gamma_\tau$ 
\item[] $\quad$ $\qquad it = it+1$
\item[] $\quad$ $\mbox{end}$
\end{itemize}

\begin{rema}\label{remarkF}
Similar to the cases for acoustic and electromagnetic waves, the elasticity
transmission eigenvalue problem is nonlinear and non-self-adjoint. Although not
theoretically justified, numerical results indicate that there exist complex
elasticity transmission eigenvalues, as we will see shortly. The above method
can compute only real eigenvalues, which correspond to the frequencies of
elasticity waves. The physical meaning of complex eigenvalues is not yet clear.
\end{rema}

%


\section{A Mixed Finite Element Method}

The method proposed above needs to compute roots of nonlinear functions and many
generalized eigenvalue eigenvalue problems \eqref{Linear}. In addition, the
$H^2$-conforming Argyris element is used for discretization. In this section, we
give a much simpler mixed method for \eqref{Nonlinear}, which also computes
complex eigenvalues. It rewrites the fourth order problem into two second order
equations. Then the simpler $H^1$-conforming Lagrange element can be applied
directly. The purpose of this section is to provide an alternative method for
verification.

Recall that the fourth order formulation of ETE is to find $\tau$ and $ {\bf 0}\ne {\boldsymbol  w} \in V$ such that
\begin{equation}\label{quadratictau}
\left(\nabla \cdot  \sigma + \tau \rho_1
\right) (\rho_1-\rho_0)^{-1} \left(\nabla \cdot \sigma +
\tau \rho_0 \right) {\boldsymbol w} = {\bf 0}.
\end{equation}
Following \cite{JiSunTurner2012ACMTOM}, we introduce an
auxiliary variable ${\boldsymbol  v}$ such that
\[
{\boldsymbol  v}=(\rho_1-\rho_0)^{-1} \left(\nabla \cdot \sigma
+ \tau \rho_0 \right) {\boldsymbol w}.
\]
The special form of ${\boldsymbol  v}$ is chosen such that the quadratic eigenvalue problem \eqref{quadratictau}
can be written as a linear eigenvalue system. Using ${\boldsymbol w}$ and the auxiliary variable ${\boldsymbol  v}$, we obtain 
\begin{eqnarray*}
\left(\nabla \cdot \sigma + \tau \rho_0 \right) {\boldsymbol w}&=&(\rho_1-\rho_0){\boldsymbol  v},\\
\left(\nabla \cdot \sigma + \tau \rho_1 \right) {\boldsymbol v}&=&0,
\end{eqnarray*}
i.e.,
\begin{eqnarray*}
(\rho_1-\rho_0){\boldsymbol  v} - \nabla \cdot \sigma( {\boldsymbol w})&=& \tau \rho_0 {\boldsymbol w},\\
- \nabla \cdot \sigma( {\boldsymbol v})&=& \tau \rho_1 {\boldsymbol v}.
\end{eqnarray*}

The weak form is to find $(\tau, {\boldsymbol w}, {\boldsymbol v})\in
\mathbb{C}\times H_0^1(D)^2\times H^1(D)^2$ such that
\begin{eqnarray*}
((\rho_1-\rho_0){\boldsymbol  v} ,{\boldsymbol  \phi} )+(\sigma
({\boldsymbol w}),\nabla {\boldsymbol  \phi} )&=&\tau(\rho_0 {\boldsymbol w},
{\boldsymbol  \phi} ) \quad \text{for all } {\boldsymbol  \phi}\in H^1(D)^2,\\
( \sigma ({\boldsymbol v}),\nabla {\boldsymbol  \psi}
)&=&\tau(\rho_1 {\boldsymbol v}, {\boldsymbol  \psi} ) \quad \text{for all }{\boldsymbol
 \psi}\in H_0^1(D)^2.
\end{eqnarray*}
Let $S_h \subset H^1(D)$ be the Lagrange finite element space and $S_h^0 \subset S_h$ with zero trace. 
The discrete problem is to find $(\tau_h, {\boldsymbol w}_h, {\boldsymbol v}_h)
\in \mathbb C \times (S_h^0)^2 \times (S_h)^2$ such that
\begin{eqnarray*}
((\rho_1-\rho_0){\boldsymbol  v}_h ,{\boldsymbol  \phi}_h
)+(\sigma ({\boldsymbol w}_h),\nabla {\boldsymbol  \phi}_h
)&=&\tau(\rho_0 {\boldsymbol w}_h, {\boldsymbol  \phi}_h ) \quad \text{for all }
{\boldsymbol  \phi}_h\in (S_h)^2,\\
(\sigma ({\boldsymbol v}_h),\nabla {\boldsymbol  \psi}_h
)&=&\tau(\rho_1 {\boldsymbol v}_h, {\boldsymbol  \psi}_h ) \quad \text{for all }
{\boldsymbol  \psi}_h\in (S_h^0)^2.
\end{eqnarray*}

\begin{rema}
Theory for \eqref{quadratictau} is not yet available and we are not able to provide a convergence analysis of the mixed
finite element method. Nonetheless, the above two methods do seem to compute the same set of real transmission eigenvalues.
\end{rema}

\section{Numerical Examples}

In this section, we present some numerical results using three domains: 
a disk with radius $R = 1/2$, the unit square and an L-shaped domain given by $(0,1)\times (0,1) \setminus [1/2, 1]\times[1/2,1]$.
Four levels of uniformly refined triangular meshes are generated for numerical experiments. The mesh size of initial mesh is $h_1=0.1$
and $h_i=h_{i-1}/2, i=2,3,4$.
Note that further refinement would lead to very large matrix eigenvalue problems which take too long to solve.
All examples are done using Matlab 2016a on a MacBook Pro with 16G memory and 3.3GHz Intel Core i7 processor.

\subsection{Results for \eqref{Linear}}

We first check the convergence rate of the Argyris method for the fourth order generalized eigenvalue problem \eqref{Linear} with fixed $\tau=2$.
Other parameters are chosen as follows
\begin{equation}\label{parameters}
\mu = 1/16, \quad \lambda = 1/4, \quad \rho_0 = 1, \quad \rho_1 = 4.
\end{equation}
The relative error is defined as
\[
E_{i+1} = \frac{|\gamma_{i+1}-\gamma_i|}{|\gamma_i|},\quad i=1,2,3,
\]
where $\gamma_{i}$ is the eigenvalue computed using the mesh with size $h_i$.
Then the convergence order is simply
\begin{align}
\mbox{convergence order}=\log_2{\frac{E_{i+1}} {E_{i+2}}},\quad \ i=1, 2.
\end{align}

The first eigenvalues are shown in Table~\ref{table1} for three domains. The convergence order for the unit square $2$.
The lower convergence order for the L-shaped domain is expected since the domain is non-convex. The convergence order for the disk
seems to be pre-asymptotic.
\begin{table}
\begin{center}
\begin{tabular}{lllllllllll}
\hline
$h$&Unit square &order & L-shape &order& Circle &order \\
\hline
0.1&1.97544798  & & 4.254621  & &2.074928  & \\
0.05&1.97544304 & & 4.244708 & &2.062945  & \\
0.025&1.97544109 & 1.533635 &  4.237900 &0.924724 &2.057076 &  1.247059\\
0.0125&1.97544043 &1.979999  &  4.233673 & 1.384228&  2.054222
& 1.612053\\
\hline
\end{tabular}
\caption{The first generalized eigenvalue of \eqref{Linear} for three domains.}
\label{table1}
\end{center}
\end{table}

\subsection{Result for $f_h(\tau)$}

We plot function $f_h(\tau)$ for three domains using the meshes with $h_3\approx 0.025$. 
For parameters given in \eqref{parameters}, the first elasticity eigenvalue $\delta_1=3.679328$ for the disk, $\delta_1=3.251402$ for the unit square, 
$\delta_1=4.325472$ for the L-shaped domain, which are computed by linear finite elements. 
According to Theorem \ref{decreasingfunc}, $f(\tau)$ is a decreasing function on $\left(0,\frac{\delta_1(\rho_0+\rho_1)}{2\rho_0\rho_1}\right)$. 
Plugging the values for $\delta_1, \rho_0, \rho_1$, we have that $f(\tau)$ is decreasing on 
\begin{equation}\label{intv}
(0,2.299580),\quad(0,2.032126), \quad \text{and}\quad (0,2.703420)
\end{equation}
for the three domains, respectively. 
In Figure \ref{fh}, we plot $f_{h}(\tau)=\gamma_{h}(\tau)-\tau$ for several eigenvalues.
The plots show that $f_h$ is decreasing right to the origin. 
In fact, it can be seen that $f_{h}(\tau)$ is decreasing on much larger intervals than those in \eqref{intv} predicted by Theorem \ref{decreasingfunc}.

\begin{figure}
\centering
\subfigure{\includegraphics[width=0.48\textwidth,height=0.36\textwidth]{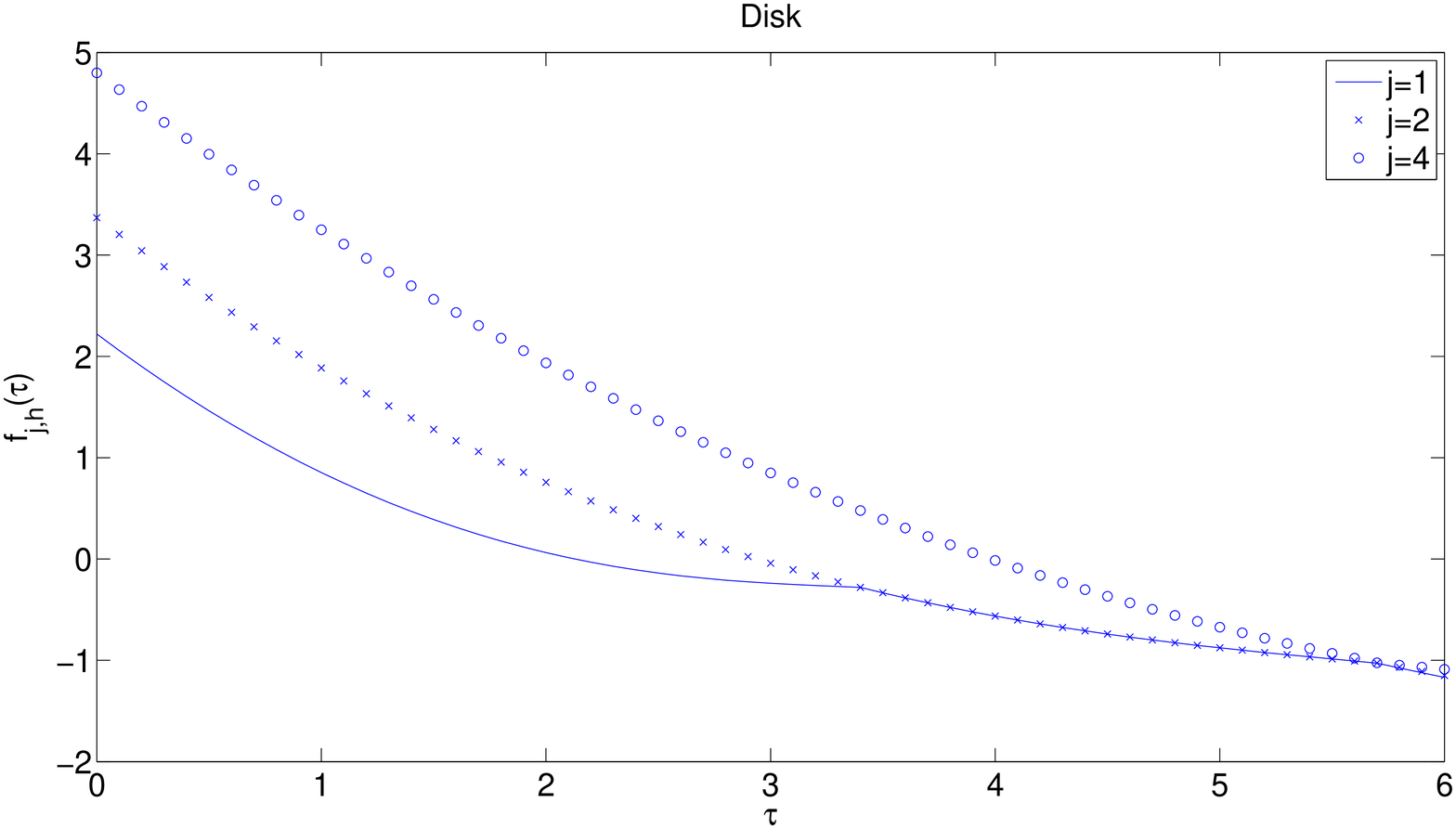}}
\subfigure{\includegraphics[width=0.48\textwidth,height=0.36\textwidth]{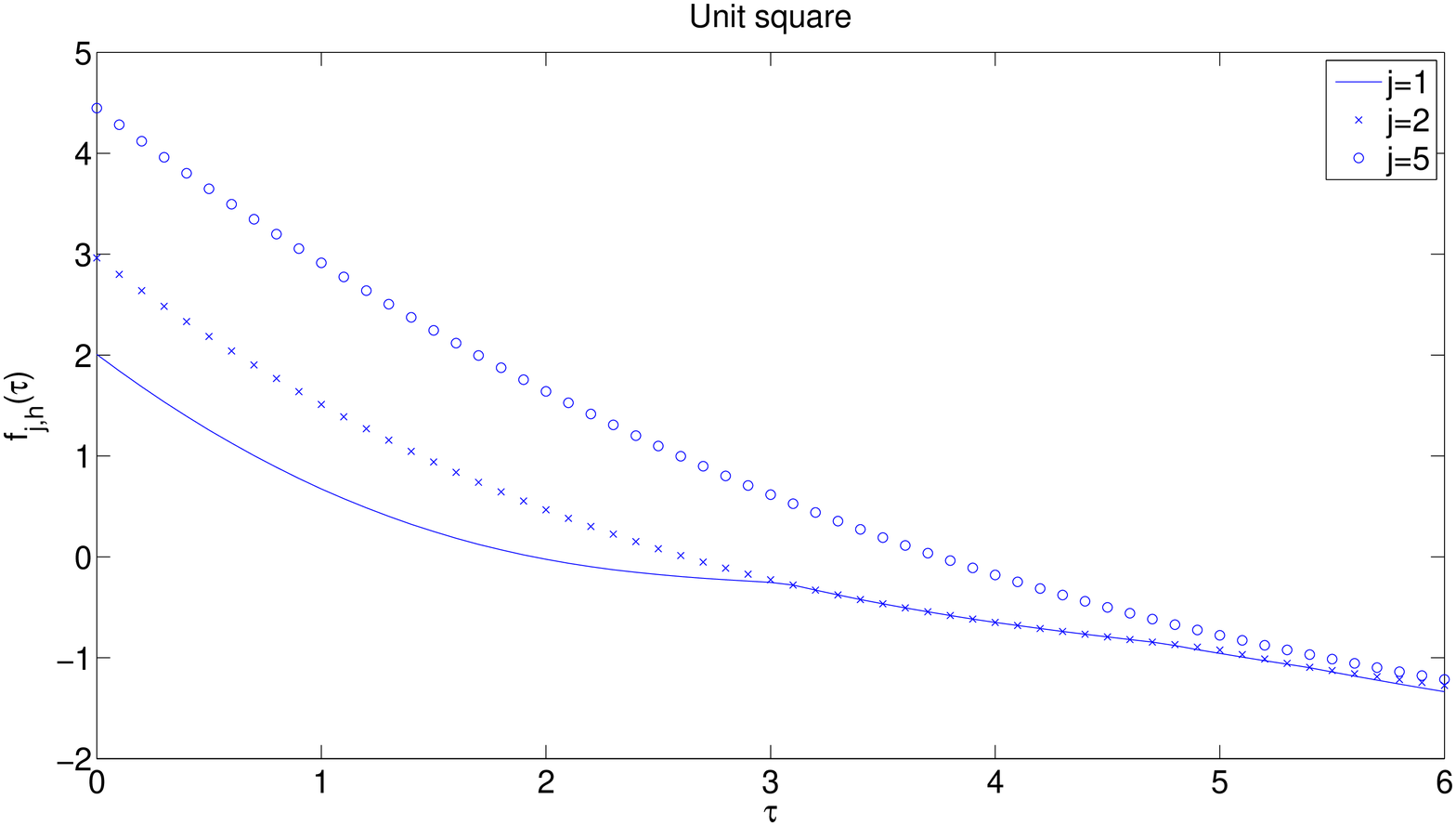}} \\
\subfigure{\includegraphics[width=0.48\textwidth,height=0.36\textwidth]{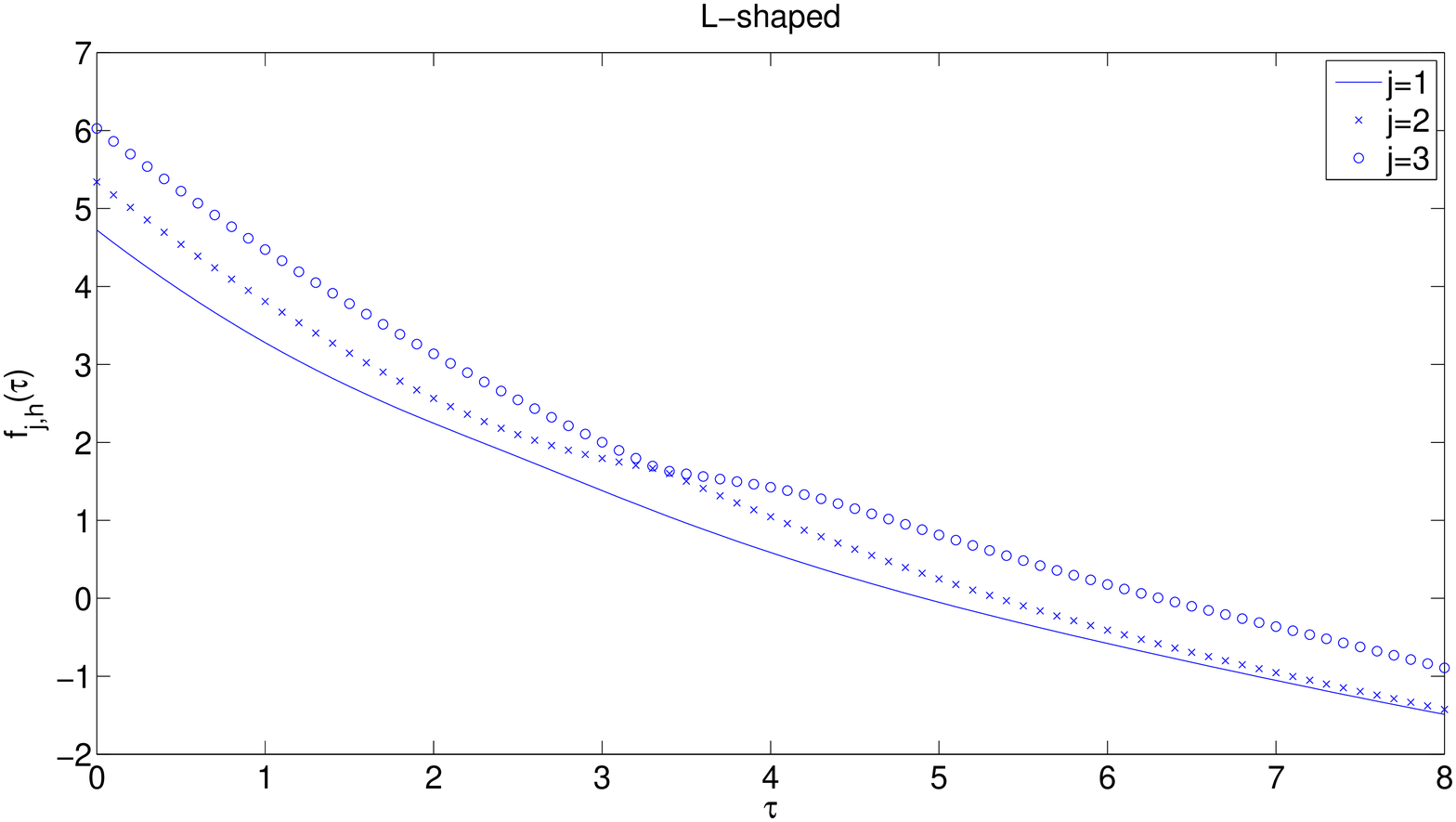}}
\caption{Plots of $f_{j,h}=\gamma_{j,h}(\tau)-\tau$ v.s. $\tau$. Top Left: the disk for $j=1,2, 4$. Top Right: the unit square for $j=1,2,5$. Bottom: the L-shaped domain for $j=1,2,3$.}
\label{fh}
\end{figure}

\subsection{Real transmission eigenvalues}

Next we present the results of the smallest transmission eigenvalues. 
Table~\ref{tableargyris} gives the computed eigenvalues and the convergence
orders of the first transmission eigenvalue of three domains using the secant
method. It can be seen that the convergence rate for the unite square is
approximately $2$ indicating that the associated eigenfunction $u \in H^3(D)$.
The convergence rate for the L-shaped domain is much lower, which is likely
caused by the low regularity of the eigenfunction. Similar results can be
observed for the biharmonic eigenvalue problem (see Chap. 4 of \cite{SunZhou2016}
or \cite{BrennerMonkSun2015LNCSE}).

\begin{table}[h!]
\begin{center}
\begin{tabular}{lllllllllll}
\hline
$h$&Unit square &order & L-shaped &order& Circle &order \\
\hline
0.1&1.94289620 & &  4.956235 & &2.151129 & \\
0.05&1.94288512& &  4.911192& &2.128551 & \\
0.025&1.94287991& 1.367706 &  4.887524&1.166121 &2.117252& 1.180424\\
0.0125& 1.94287811 &  1.961417& 4.874986& 1.529921&  2.110723 & 1.449210\\
\hline
\end{tabular}
\caption{Convergence orders of the first real transmission eigenvalue using the secant method.}
\label{tableargyris}
\end{center}
\end{table}
\begin{rema}
When the domain is a disk, one can obtain transmission eigenvalues associated
with radially-symmetric eigenfunctions using separation of variables. The detail
of derivation of such eigenvalues is given in Appendix~A.
\end{rema}

In Table~\ref{table2},
we show the six smallest transmission eigenvalues computed by the secant method for the three domains with mesh size $h \approx0.025$. 
The method converges in a few steps.
\begin{table}
\begin{center}
\begin{tabular}{lllllllllll}
\hline
$\tau_j$&Unit square  &NOI  & L-shaped &NOI& Circle &NOI \\
\hline
$\tau_1$&1.942885 &7& 4.887524 & 6&2.117252 & 7\\
$\tau_2$&2.618883 & 7&  5.314739& 5&2.921413 & 9\\
$\tau_3$&2.618883 &8&  6.301718&6 &2.921413& 12\\
$\tau_4$&3.247320 & 6&  6.664844&  6& 3.958629&8\\
$\tau_5$&3.748613 & 5&  7.418109&  7& 3.958629 &8\\
$\tau_6$&4.418714 & 5&  8.021953&  10& 5.175589 &6\\
\hline
\end{tabular}
\caption{Six smallest transmission eigenvalues computed by the secant method. "NOI" denotes the number of iterations.}
\label{table2}
\end{center}
\end{table}
In Table~\ref{table3},
we give the six smallest transmission eigenvalues computed by the secant method for the three domains using a different set of parameters
$\mu = 1/5, \lambda = 1/5, \rho_0 = 1/20, \rho_1 = 3$.  The iteration converges in a few steps as well.
\begin{table}
\begin{center}
\begin{tabular}{lllllllllll}
\hline
$\tau_j$&Unit square  &NOI  & L-shaped &NOI& Circle &NOI \\
\hline
$\tau_1$&6.451568 &4& 10.882301& 4&7.130154 & 4\\
$\tau_2$&7.649225 & 5&  13.680402& 4& 9.062169& 5\\
$\tau_3$& 7.649225 &5&  16.694660&4&9.063492& 8\\
$\tau_4$&11.201158 & 8&  18.281319&  5& 12.781042&6\\
$\tau_5$&11.404597 & 4&  19.550610&  4& 12.782461 &8\\
$\tau_6$&12.099399 & 6&  21.397458&  4& 14.216742 &4\\
\hline
\end{tabular}
\caption{Six smallest transmission eigenvalues with $\mu = 1/5, \lambda = 1/5, \rho_0 = 1/20, \rho_1 = 3$, "NOI" denotes the number of iterations.}
\label{table3}
\end{center}
\end{table}

\subsection{The special case of disk}

From Appendix A, a radially-symmetric transmission eigenvalue of the disk is the first root of $Z_0$ defined in \eqref{eqZ0}.
Using some root finding technique, we find that $\omega=3.554954$, i.e., $\tau=12.637700$. However, it is not the smallest transmission eigenvalue of the disk.
The secant method computes an approximate transmission eigenvalue $\tau=12.662693$ with $h=0.05$ and $\tau=12.624538$ with $h=0.025$. 
The mixed method also computes a transmission eigenvalue $\tau=12.713678$ with $h=0.025$.
Figure \ref{eigfunction} plots the eigenfunction ${\boldsymbol u}$ associated with this eigenvalue, which appear to be radially-symmetric.
\begin{figure}
\centering
\subfigure{\includegraphics[width=0.3\textwidth,height=0.25\textwidth]{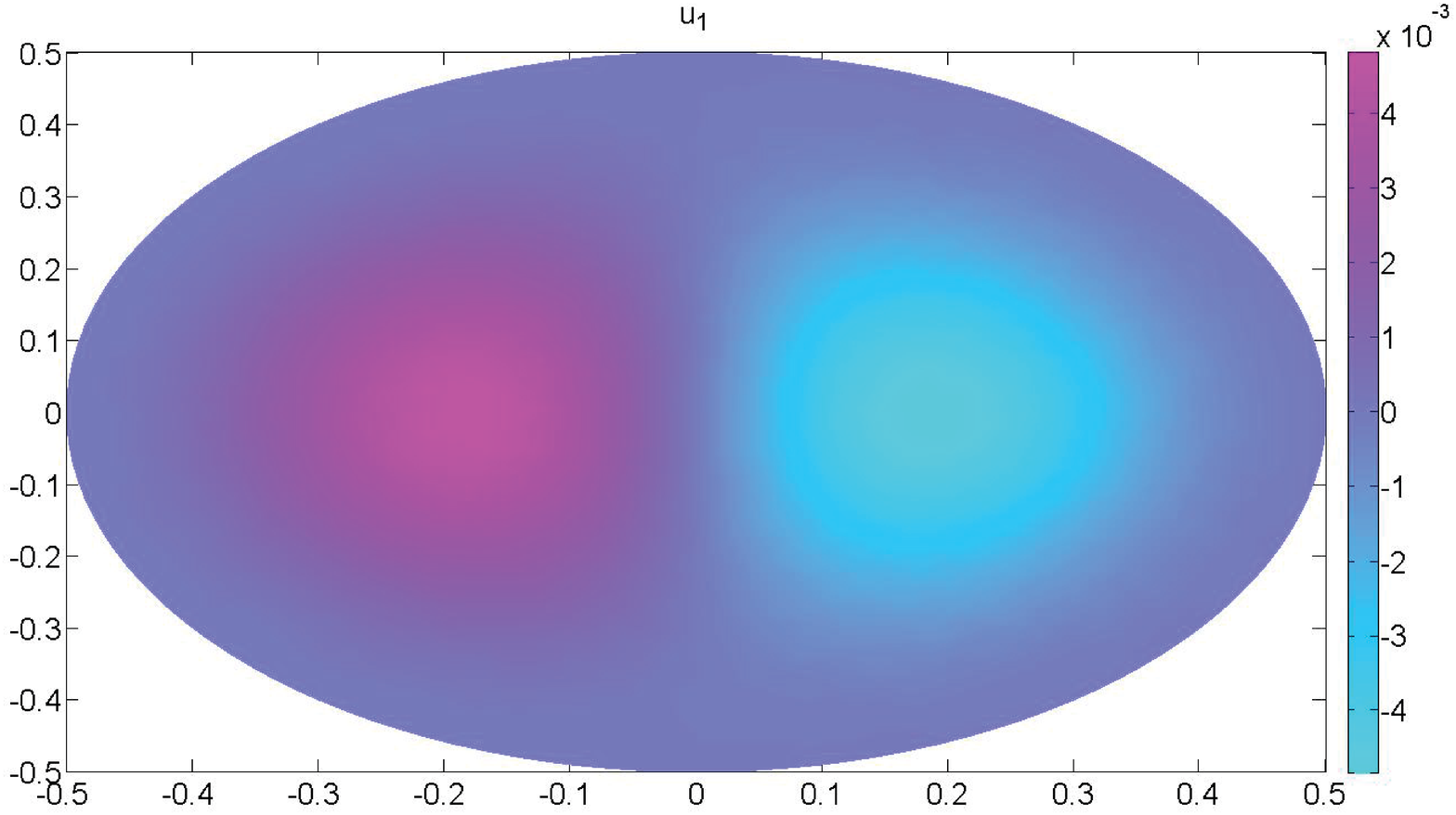}}
\subfigure{\includegraphics[width=0.3\textwidth,height=0.25\textwidth]{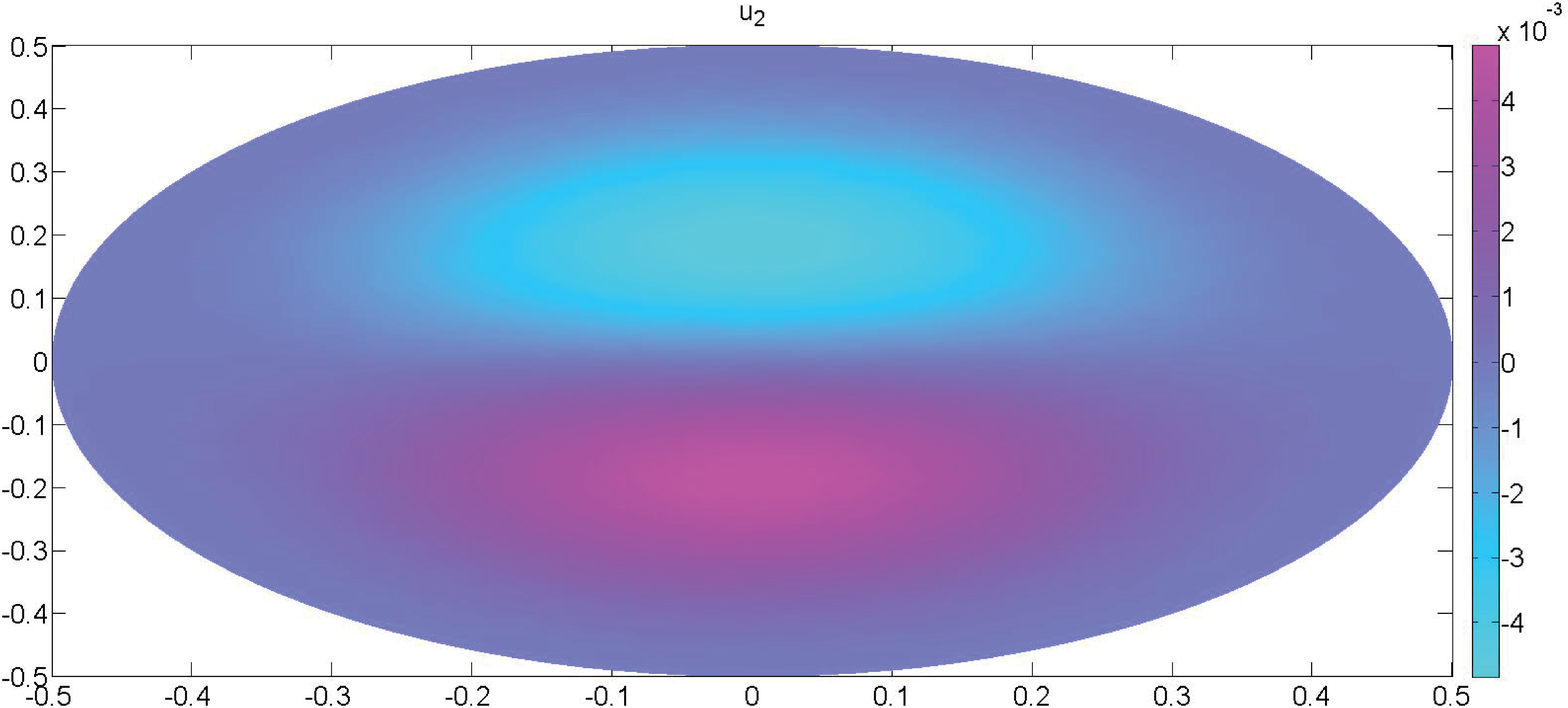}}
\subfigure{\includegraphics[width=0.3\textwidth,height=0.25\textwidth]{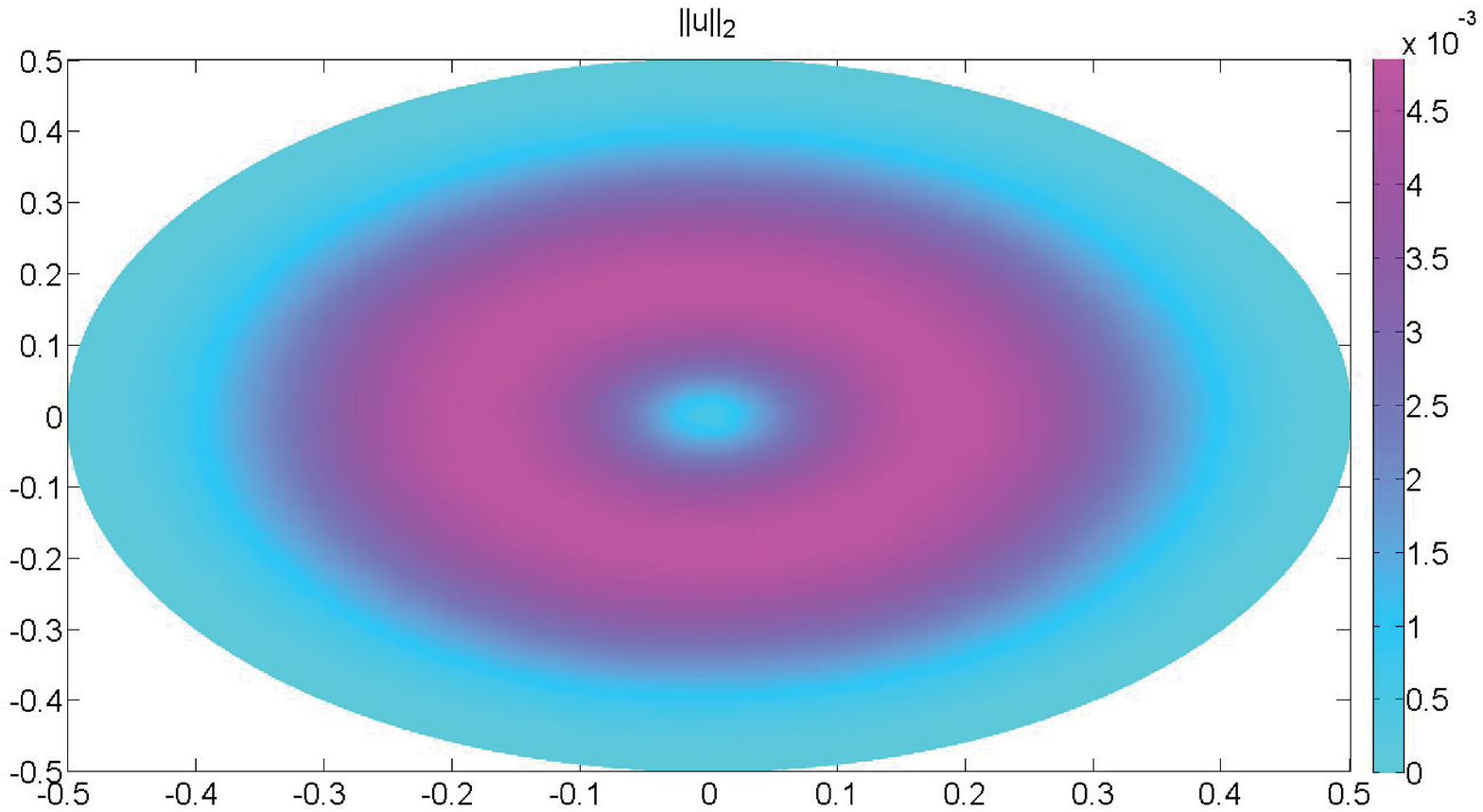}}
\caption{A radially-symmetric eigenfunction. Left: $u_1$. Middle: $u_2$. Right: $|{\boldsymbol u}=(u_1,u_2)|$.}
\label{eigfunction}
\end{figure}
Note that not all eigenfunctions are radially-symmetric. Figure \ref{secondeig} is the eigenfunction associated with the second eigenvalue. Clearly, it is not a radially-symmetric function.
\begin{figure}
\centering
\subfigure{\includegraphics[width=0.3\textwidth,height=0.25\textwidth]{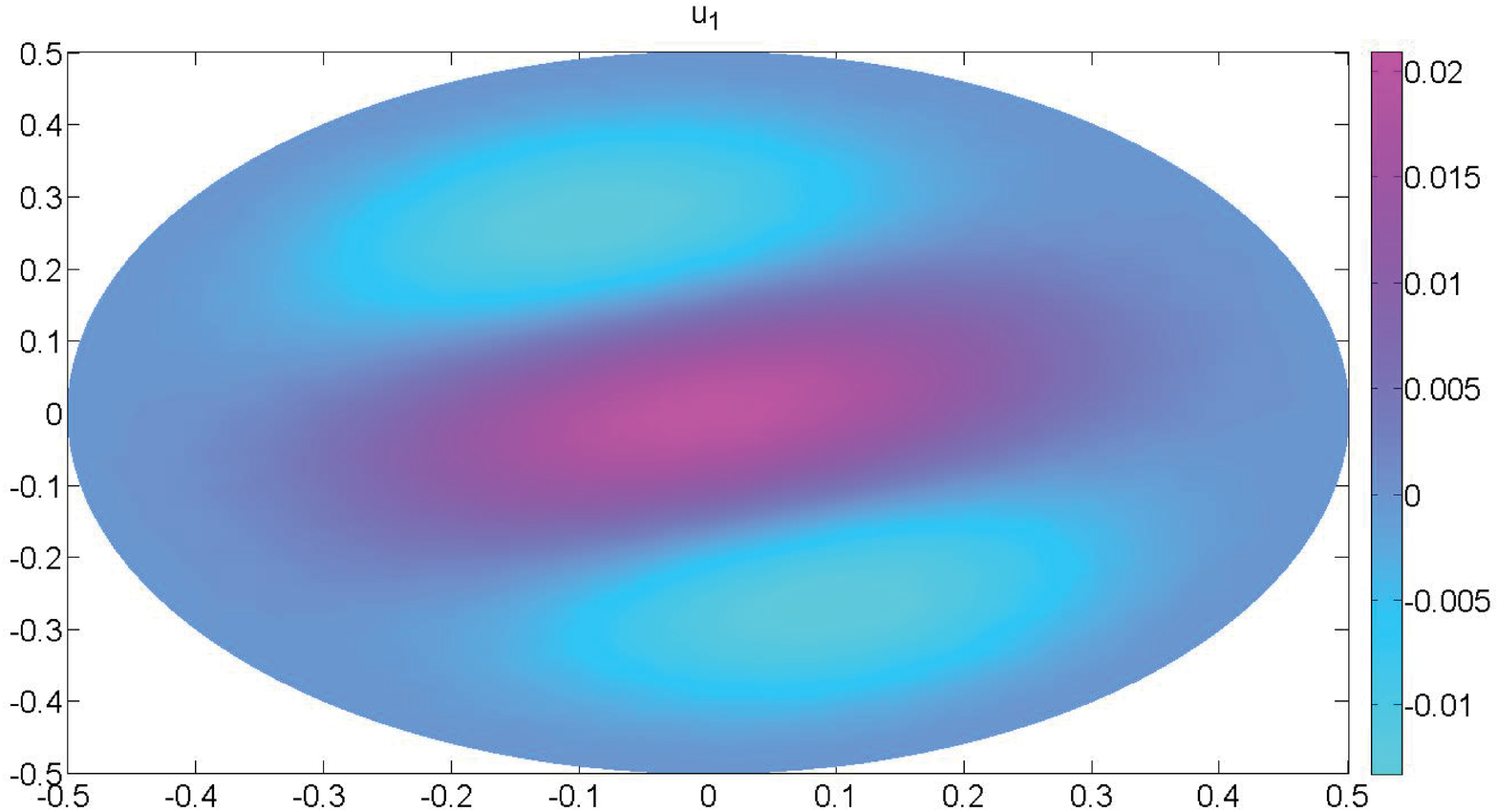}}
\subfigure{\includegraphics[width=0.3\textwidth,height=0.25\textwidth]{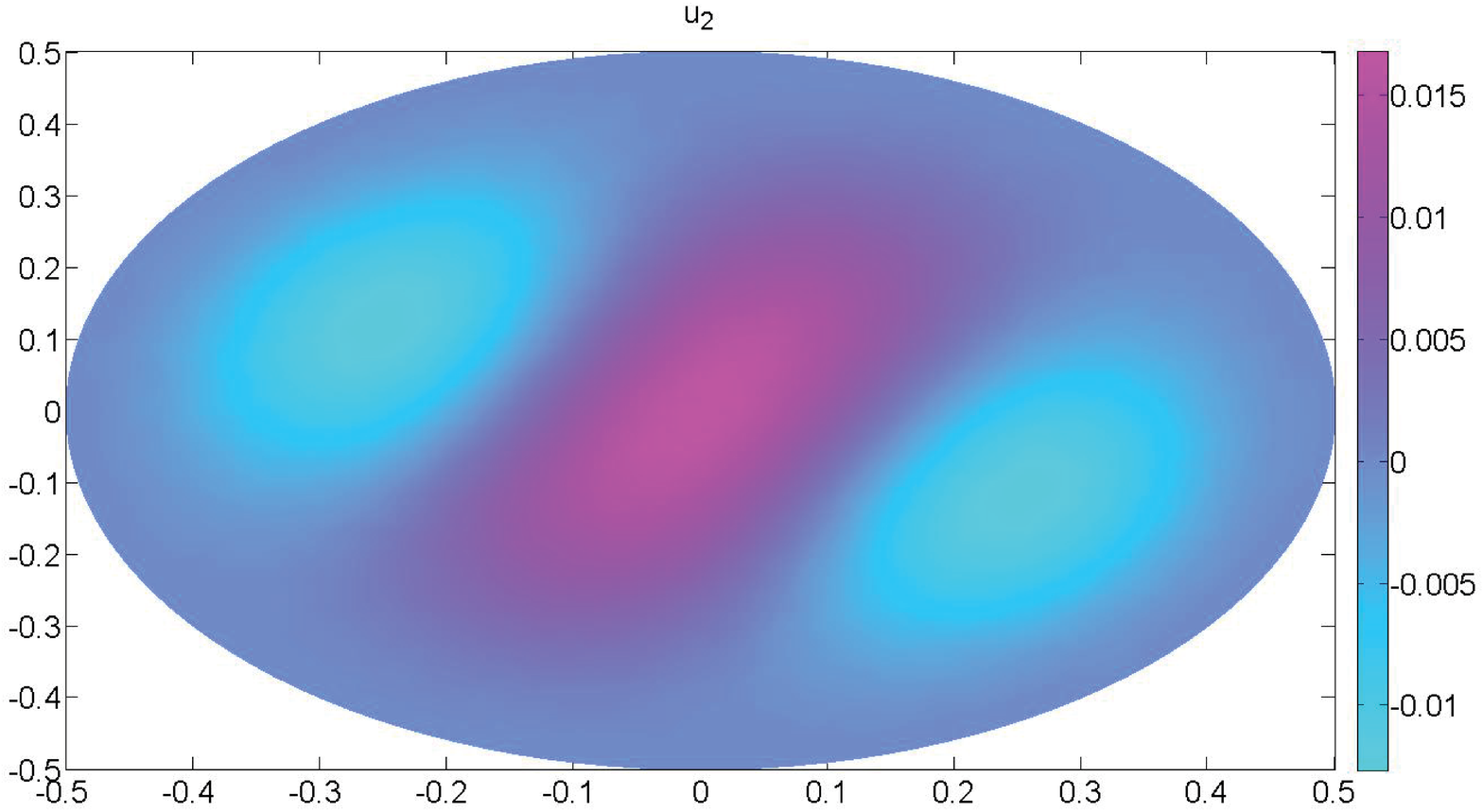}}
\subfigure{\includegraphics[width=0.3\textwidth,height=0.25\textwidth]{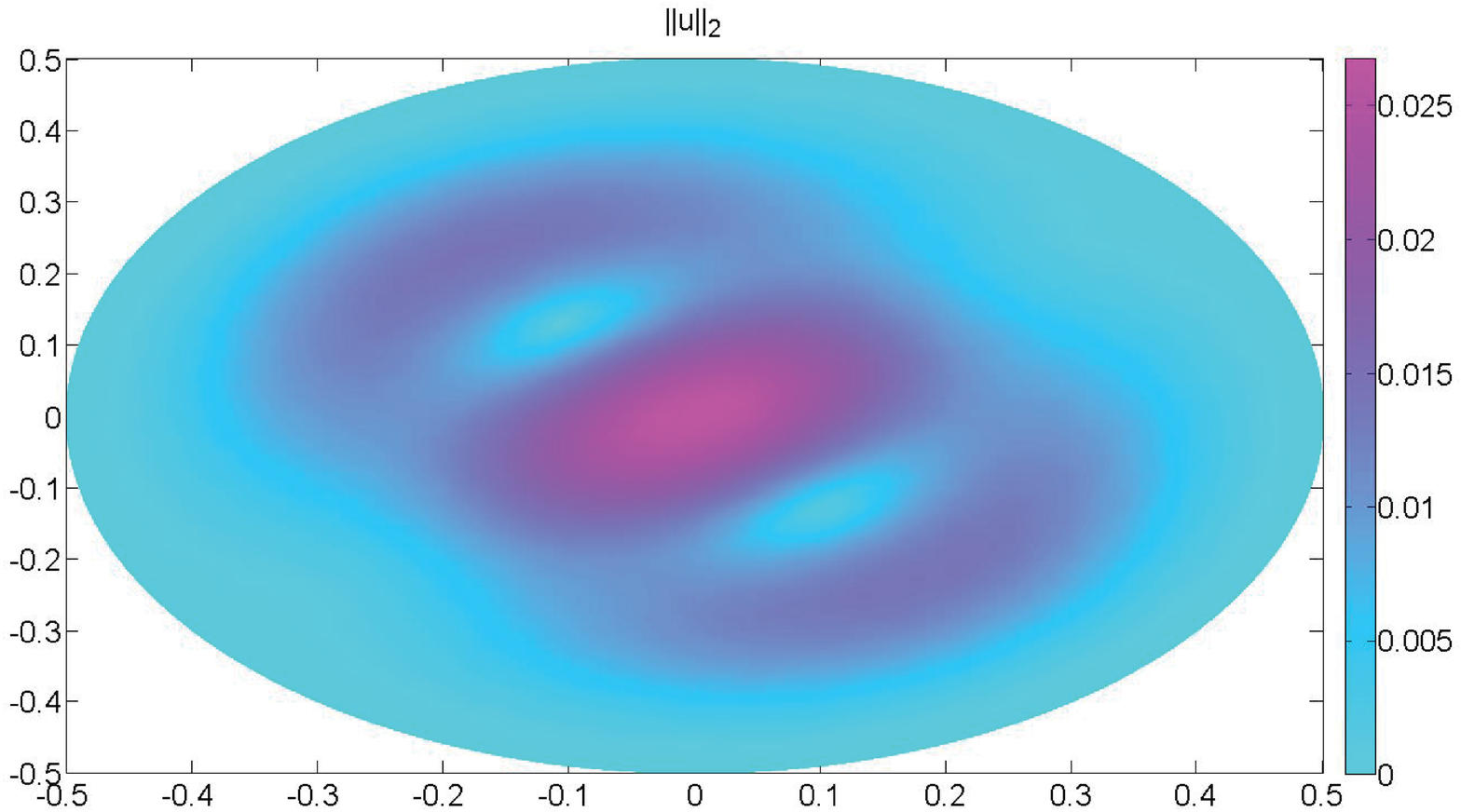}}
\caption{Second eigenfunction. Left: $u_1$. Middle: $u_2$. Right: $|{\boldsymbol u}=(u_1,u_2)|$.}
\label{secondeig}
\end{figure}

\subsection{Results by the mixed method}

For comparison and verification, we also computes the transmission eigenvalues using the mixed method. 
Table~\ref{tablemix} shows the first transmission eigenvalue and convergence rates.
\begin{table}
\begin{center}
\begin{tabular}{lllllllllll}
\hline
$h$&Unit square &order & L-shaped &order& Circle &order \\
\hline
0.1&2.393618 & & 7.117873 & &2.734748 & \\
0.05&2.040967& & 5.466449& &2.221626 & \\
0.025&1.967283& 2.273017 &  5.028595&1.983289 &2.136957& 2.527205\\
0.0125&1.948971&2.328781 &  4.907390& 2.205551&  2.113930 & 2.225566\\
\hline
\end{tabular}
\caption{The first real transmission eigenvalue of the mixed method.}
\label{tablemix}
\end{center}
\end{table}
Note that ETE is non-self-adjoint and the secant method only computes real
transmission eigenvalues. The mixed method can compute complex eigenvalues.
Table~\ref{tablemixc} gives the convergence orders of the first complex
transmission eigenvalue.
\begin{table}
\begin{center}
\begin{tabular}{lllllllllll}
\hline
$h$&Unit square &order & L-shaped &order& Circle &order \\
\hline
0.1&3.335717 - 3.171243i & & 3.631454 - 3.333002i  & &4.153116 - 0.863572i  & \\
0.05&3.494072 - 1.117876i& & 3.616318 - 3.116377i& &3.963564 - 1.126331i  & \\
0.025&3.422905 - 1.097453i & 4.450560 &  3.613011 - 3.061526i &2.049651 &3.897421 - 1.131189i&  2.137062\\
0.0125&3.402856 - 1.090597i&2.167773  &  3.611227 - 3.047409i &  2.280984&  3.876819 - 1.126730i
& 2.040988\\
\hline
\end{tabular}
\caption{The first complex transmission eigenvalue of the mixed method.}
\label{tablemixc}
\end{center}
\end{table}

\section{Conclusions and Future Work}

In this paper, we propose an iterative method to compute a few smallest
transmission eigenvalues for elastic waves. The major advantage of
this method is the accuracy and effectiveness since we only need to compute a few
eigenvalues of Hermitian eigenvalue problems instead of computing the full eigensystem
of a non-Hermitian eigenvalue problem. This fits the practical need in the sense
that in general only the smallest transmission eigenvalues are needed for
the estimation of the elastic properties of the material. We prove the
convergence of the proposed method. The effectiveness of the method is
verified by some numerical examples. We have also given a mixed method without
proof. In future, we will consider the complex eigenvalues of the elastic
waves. The analytic property of the function $f(\tau) = \gamma(\tau)-\tau$
is also an interesting topic.

\section*{Appendix A: Radially Symmetric Case on Disks}

We derive the equation satisfied by a transmission eigenvalue whose associated eigenfunction is radially symmetric on a disk.
Let $D \subset \mathbb R^2$ be a disk with radius $R$.
Let ${\boldsymbol u}=(w, v)^\top$. Writing the elasticity wave equation
\eqref{ElasticityLHS1}
component wise, we have that
\begin{eqnarray} \label{2mulambda}
(2\mu + \lambda) \frac{\partial^2 w}{\partial x_1^2}+(\lambda + \mu)\frac{\partial^2 v}{\partial x_2 \partial x_1}
				+ \mu \frac{\partial^2 w}{\partial^2 x_2}+\omega^2 \rho w = 0, \\
\label{mulambda}	   (\mu+\lambda)\frac{\partial^2 w}{\partial x_2 \partial x_1}+\mu \frac{\partial^2 v}{\partial x_1^2} +
	   (2\mu+\lambda) \frac{\partial^2 v}{\partial x_2^2}+\omega^2 \rho v = 0.
\end{eqnarray}
If we consider the solution in the form of radially-symmetric vector field
${\boldsymbol u}({\boldsymbol x}) = u(r){\boldsymbol e}_r$, where
$r=|{\boldsymbol x}|$ and ${\boldsymbol e}_r = {\boldsymbol x}/{r}$, $w = u(r)\cos \theta, v = u(r)\sin \theta$,
\eqref{2mulambda} can be written as
\[
(\mu + \lambda) \left( \frac{\partial^2 w}{\partial x_1^2}+\frac{\partial^2 v}{\partial x_2 \partial x_1}\right)
				+ \mu \left(\frac{\partial^2 w}{\partial^2 x_2}+ \frac{\partial^2 w}{\partial x_1^2}\right) +\omega^2 \rho w =0.
\]
Using polar coordinate, \eqref{2mulambda} becomes
\[
(\lambda+\mu) \left( u_{rr}  +\frac{1}{r} u_r  -\frac{1}{r^2} u   \right)
+\mu\left(u_{rr} + \frac{1}{r} u_r - \frac{1}{r^2}u\right) +\omega^2 \rho u = 0,
\]
i.e.,
\[
(\lambda+2\mu) \left( u_{rr}  +\frac{1}{r} u_r  -\frac{1}{r^2} u  
\right)+\omega^2 \rho u = 0.
\]
Similarly, \eqref{mulambda} is simply
\[
(\mu+\lambda)\left(u_{rr} + \frac{1}{r}u_r -\frac{1}{r^2} u\right)+
			\mu \left(u_{rr} + \frac{1}{r}u_r - \frac{1}{r^2}u\right)
	  +\omega^2 \rho u = 0,
\]
i.e.,
\[
(2\mu+\lambda)\left(u_{rr} + \frac{1}{r}u_r -\frac{1}{r^2} u\right) +\omega^2 \rho u = 0.
\]
The above equation can be written as
\[
r^2 \frac{{\rm d}^2 u}{{\rm d}r^2}+r \frac{{\rm d}u}{{\rm d}r} + \left( r^2
\frac{\omega^2 \rho}{2\mu+\lambda}-1 \right) u = 0.
\]
The solution of the above equation is given by the Bessel function of order one
$J_1(ar)$, where
\[
a = \omega \sqrt{\frac{\rho}{2\mu + \lambda}}.
\]
Then we obtain that ${\boldsymbol u}=(w, v)^\top:=(J_1(ar)\cos \theta,
J_1(ar) \sin \theta)^\top$.

Next we look at the boundary condition involving $\sigma({\boldsymbol u}) {\boldsymbol \nu}$. For the
transmission eigenvalue problem, we assume that
\[
{\boldsymbol u} = \begin{pmatrix} J_1(a_1 r) \cos \theta \\ J_1(a_1 r) \sin \theta \end{pmatrix}, \quad
{\boldsymbol v} = C\begin{pmatrix} J_1(a_2 r) \cos \theta \\ J_1(a_2 r) \sin \theta \end{pmatrix},
\]
where $C$ is a constant to be determined from \eqref{tep} and
\[
a_1 = \omega \sqrt{\frac{\rho_1}{2\mu + \lambda}}, \quad a_2 = \omega \sqrt{\frac{\rho_2}{2\mu + \lambda}}.
\]

Note that ${\boldsymbol \nu} = (\nu_1, \nu_2)^\top = (\cos \theta, \sin
\theta)^T$. It follows from \eqref{DefSigma} that the first component of
$\sigma({\boldsymbol u}){\boldsymbol \nu}$ is
\begin{eqnarray*}
\lambda \left[ \frac{\partial }{\partial r}J(ar) + \frac{1}{r} J(ar) \right]
\cos \theta  + 2\mu \frac{\partial }{\partial r}J(ar) \cos \theta.
\end{eqnarray*}
The second component of $\sigma({\boldsymbol u}){\boldsymbol \nu}$ is
\begin{eqnarray*}
 2\mu \frac{\partial }{\partial r}J(ar)  \sin \theta + \lambda \left[ 
\frac{\partial }{\partial r}J(ar) + \frac{1}{r} J(ar) \right] \sin \theta.
\end{eqnarray*}

Using the boundary conditions \eqref{tep}, we obtain
\[
C=\frac{J_1(a_1 R)}{J_1(a_2 R)},
\]
and
\[
\left(2\mu \frac{\partial }{\partial r}J(a_1r)  + \lambda \left[  \frac{\partial }{\partial r}J(a_1r) + \frac{1}{r} J(a_1r) \right] \right) \Big\vert_{r=R}
= C\left(2\mu \frac{\partial }{\partial r}J(a_2r)  + \lambda \left[  \frac{\partial }{\partial r}J(a_2r) + \frac{1}{r} J(a_2r) \right] \right) \Big \vert_{r=R}.
\]
Hence $\omega$ is a transmission eigenvalue if it satisfies
\begin{equation}\label{eqZ0}
Z_0(\omega):=\left | \begin{array}{cc} J_1\left( a_1R\right) & J_1\left( a_2 R\right)  \\
a_1J_1'(a_1 R)  & a_2 J_1'(a_2 R)  \end{array}\right|=0.
\end{equation}
Figure \ref{absZ0s} is the contour plot of $|Z_0(\omega)|$ on the complex plane.

\begin{figure}
\begin{center}
{ \scalebox{0.3} {\includegraphics{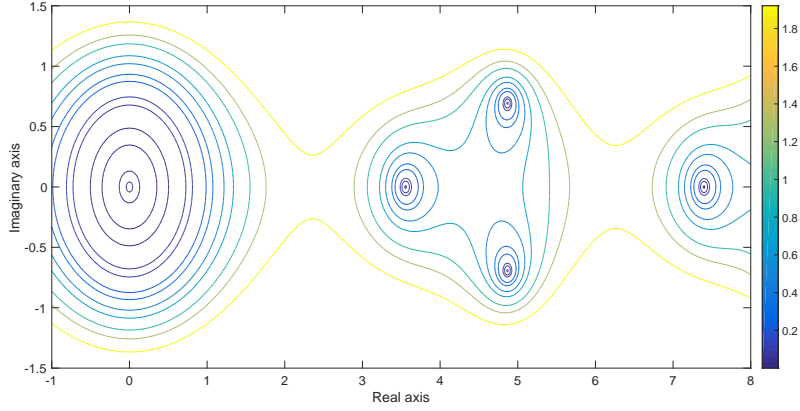}}}
\caption{The contour plot of $|Z_0(\omega)|$ with $\mu = 1/16, \lambda = 1/4,
\rho_0 = 1, \rho_1 = 4$. The centers of the circular curves indicate the
locations of transmission eigenvalues.}
 \label{absZ0s}
\end{center}
\end{figure}

\section*{Appendix B: Imposing Boundary Conditions}\label{AppendixB}

The boundary condition for $V$ needs careful treatment for the Argyris
element. On one boundary of a triangle $\gamma_e \subset \mathcal{T} $ with unit
outward normal $(n_x,n_y)$ and unit tangential vector $(\tau_x,\tau_y)$, 
we consider the case $n_x n_y\neq 0$. Otherwise it is easy to figure out. It is clear that
\begin{equation}\label{relation}
\tau_xn_x+\tau_yn_y=0,\quad \tau_x^2+\tau_y^2=1,\quad n_x^2+n_y^2=1.
\end{equation}
On the boundary $\gamma_e$, $\boldsymbol u=(u_1,u_2)^\top=\boldsymbol{0}$, so
the tangent derivatives are also $\boldsymbol{0}$, i.e.,
\[
\frac{\partial u_1}{\partial x}\tau_x+\frac{\partial u_1}{\partial y}\tau_y=0,\quad \frac{\partial u_2}{\partial x}\tau_x+\frac{\partial u_2}{\partial y}\tau_y=0.
\]
The boundary condition $\sigma(\boldsymbol{u}){\boldsymbol \nu}=\boldsymbol{0}$
means
\begin{eqnarray*}
((\lambda+2\mu)\partial_x u_1 + \lambda \partial_y u_2)n_x+ \mu(\partial_y u_1 + \partial_x u_2)n_y=0,\\
\mu (\partial_x u_2 + \partial_y u_1)n_x+(\lambda\partial_x u_1 +(\lambda+2\mu)\partial_y u_2)n_y=0.
\end{eqnarray*}
Substituting
\[
\frac{\partial u_1}{\partial y}=-\frac{\tau_x}{\tau_y}\frac{\partial
u_1}{\partial x},\quad \frac{\partial u_2}{\partial
y}=-\frac{\tau_x}{\tau_y}\frac{\partial u_2}{\partial x},
\]
into above equations, we have
\begin{eqnarray*}
(\lambda+2\mu)\partial_x u_1n_x -\frac{\tau_x}{\tau_y}\lambda \partial_x u_2 n_x -\frac{\tau_x}{\tau_y}\mu\partial_x u_1n_y + \mu\partial_x u_2n_y=0,\\
\mu\partial_x u_2 n_x-\frac{\tau_x}{\tau_y}\mu\partial_x u_1 n_x+\lambda\partial_x u_1n_y-\frac{\tau_x}{\tau_y}(\lambda+2\mu)\partial_x u_2n_y=0,
\end{eqnarray*}
i.e.,
\begin{eqnarray*}
\partial_x u_1[n_x(\lambda+2\mu)-\mu \frac{\tau_x}{\tau_y}n_y]+\partial_x u_2[\mu n_y-\frac{\tau_x}{\tau_y}\lambda n_x]=0,\\
\partial_x u_1[\lambda n_y-\frac{\tau_x}{\tau_y}\mu n_x]+\partial_x u_2[\mu
n_x-(\lambda+2\mu)n_y\frac{\tau_x}{\tau_y}]=0,
\end{eqnarray*}
which yields 
\begin{eqnarray*}
\partial_x u_1[\tau_yn_x(\lambda+2\mu)-\mu\tau_xn_y]-\partial_x u_2(\lambda+\mu)n_x\tau_x=0,\\
-\partial_x u_1(\lambda+\mu)n_x\tau_x+\partial_x u_2[\mu n_x \tau_y-(\lambda+2\mu)n_y\tau_x]=0.
\end{eqnarray*}
Here we have used the fact that $\tau_xn_x+\tau_yn_y=0$. Using \eqref{relation}, the
determination of the above equations is
\begin{eqnarray*}
&&\mu(\lambda+2\mu)n_x^2(1-\tau_x^2)+\mu(\lambda+2\mu)\tau_x^2(1-n_x^2)+(\lambda+2\mu)^2n_x^2\tau_x^2+\mu^2n_x^2\tau_x^2-(\lambda+\mu)^2n_x^2\tau_x^2 \nonumber\\
&=&\mu(\lambda+2\mu)(n_x^2+\tau_x^2)+n_x^2\tau_x^2((\lambda+2\mu)^2+\mu^2-(\lambda+\mu)^2-2\mu(\lambda+2\mu))\nonumber\\
&=&\mu(\lambda+2\mu)(n_x^2+\tau_x^2)>0.
\end{eqnarray*}
Consequently,
\[
\partial_x u_1=\partial_y u_1=\partial_x u_2=\partial_y u_2=0.
\]
The tangential derivatives of all the first-order derivatives are 0. Taking
$u_1$ for example, we have
\begin{eqnarray*}
\frac{\partial^2 u_1}{\partial x^2}\tau_x+\frac{\partial^2 u_1}{\partial x \partial y}\tau_y=0,\\
\frac{\partial^2 u_1}{\partial x \partial y}\tau_x+\frac{\partial^2 u_1}{\partial y^2}\tau_y=0,
\end{eqnarray*}
which implies
\begin{equation}\label{secondderi}
\frac{\partial^2 u_1}{\partial x^2}=-\frac{\tau^2_y}{\tau^2_x}\frac{\partial^2
u_1}{\partial y^2}, \quad \frac{\partial^2 u_1}{\partial x \partial
y}=-\frac{\tau_y}{\tau_x}\frac{\partial^2 u_1}{\partial y^2}.
\end{equation}
If a point belonging to two edges $\gamma_{e1}$ and $\gamma_{e2}$ satisfies \eqref{secondderi} with different tangential derivatives, we have
\[
\frac{\partial^2 u_1}{\partial x^2}=\frac{\partial^2 u_1}{\partial y^2}=\frac{\partial^2 u_1}{\partial x \partial y}=0.
\]


\end{document}